\documentclass[11pt]{amsart}
\usepackage{
  amsmath,
  amsfonts,
  amssymb,
  mvsymb,   
  graphicx, 
  times,
  color,
  hyphenat,
  pinlabel}
\input xy
\xyoption{all}
\newcommand{\cH}{\mathcal{H}}

\newcommand{\brak}[1]{\langle #1\rangle}
\newcommand{\n}{\noindent}
\newcommand{\foam}{\mathbf{Foam}_{/\ell}}
\newcommand{\wfoam}{\widehat{\mathbf{Foam}_{/\ell}}}
\newcommand{\kom}{\mathbf{Kom}}
\newcommand{\mf}[1]{\mathbf{MF}_{#1}}
\newcommand{\hmf}[1]{\mathbf{HMF}_{#1}}

\newcommand{\Link}{\mathbf{Link}}
\newcommand{\Modbg}{\mathbf{Mod_{bg}}}
\newcommand{\Modgr}{\mathbf{Mod_{gr}}}

\DeclareMathOperator{\End}{End}
\DeclareMathOperator{\Ext}{Ext}
\DeclareMathOperator{\HKR}{HKR}
\DeclareMathOperator{\KR}{KR}
\DeclareMathOperator{\hy}{H}

\newcommand{\figins}[3] 
{\raisebox{#1pt}{\includegraphics[height=#2 in]{figs/#3}}}
\newcommand{\figwins}[3] 
{\raisebox{#1pt}{\includegraphics[width=#2 in]{figs/#3}}}
\newcommand{\figwhins}[4] 
{\raisebox{#1pt}{\includegraphics[height=#2 in, width=#3 in]{figs/#4}}}
\newcommand{\figs}[2] 
{{\includegraphics[scale =#1]{figs/#2}}}

\newtheorem{thm}{Theorem}[section]
\newtheorem{lem}[thm]{Lemma}
\newtheorem{cor}[thm]{Corollary}
\newtheorem{prop}[thm]{Proposition}

\theoremstyle{definition}
\newtheorem{defn}[thm]{Definition}

\newcommand{\bZ}{\mathbb{Z}}
\newcommand{\bQ}{\mathbb{Q}}

\newcommand{\bC}{\mathbb{C}}

\DeclareMathOperator{\Hom}{Hom}
\DeclareMathOperator{\Ker}{Ker}

\DeclareMathOperator{\Image}{Im}
\DeclareMathOperator{\id}{Id}

\newcommand{\ra}{\rightarrow}
\newcommand{\xra}[1]{\xrightarrow{#1}}
\newcommand{\lra}{\longrightarrow}

\newcommand{\ket}[1]{|\,#1\,\rangle}

\title{The foam and the matrix factorization $sl_3$ link homologies are equivalent}
\author{Marco Mackaay}
\address{Departamento de Matem\'{a}tica\\ Universidade do Algarve\\ 
Campus de Gambelas\\ 8005-139 Faro\\ Portugal and CAMGSD\\Instituto Superior T\'{e}cnico\\ Avenida Rovisco Pais\\ 
1049-001 Lisboa\\ Portugal}
\email{mmackaay@ualg.pt}
\author{Pedro Vaz}
\address{Departamento de Matem\'{a}tica\\ Universidade do Algarve\\ 
Campus de Gambelas\\ 8005-139 Faro\\ Portugal and  
CAMGSD\\Instituto Superior T\'{e}cnico\\ Avenida Rovisco Pais\\ 
1049-001 Lisboa\\ Portugal}
\email{pfortevaz@ualg.pt}
\begin{document}
%
\begin{abstract}
We prove that the universal rational $sl_3$ link homologies which were constructed 
in \cite{khovanovsl3, mackaay-vaz}, using foams, and in \cite{KR}, using 
matrix factorizations, are naturally isomorphic as projective functors from the 
category of links and link cobordisms to the category of bigraded vector spaces. 
\end{abstract}
\maketitle
%
%
%
%
\section{Introduction}\label{sec:intro}

In~\cite{khovanovsl3} Khovanov constructed a bigraded integer link homology 
categorifying the $sl_3$ link polynomial. His construction used singular cobordisms 
called foams. Working in a category of foams modulo certain relations, 
the authors in \cite{mackaay-vaz} generalized Khovanov's theory and constructed 
the universal integer $sl_3$-link homology (see also  
\cite{morrison-nieh} for a slightly different approach). 
In~\cite{KR} Khovanov and Rozansky (KR) 
constructed a rational bigraded theory that categorified the $sl_n$ link polynomial 
for all $n>0$. They conjectured that their theory is isomorphic to the one 
in~\cite{khovanovsl3} for $n=3$, after tensoring the latter with $\bQ$. 
Their construction uses matrix factorizations and can be generalized to give 
the universal rational link homology for all $n>0$ 
(see \cite{gornik, rasmussen-diff, Wu}). In this paper we prove that the 
universal rational KR link homology for $n=3$ is equivalent to the foam link homology 
in~\cite{mackaay-vaz} tensored with $\bQ$.

One of the main difficulties one encounters when trying to relate both theories 
mentioned above is that the foam approach uses ordinary {\em webs}, which are ordinary 
oriented trivalent graphs, whereas the KR theory uses {\em KR-webs}, which are 
trivalent graphs containing two types of edges: oriented edges and unoriented 
thick edges. In general there are 
several KR-webs that one can associate to an ordinary web, so there is no 
obvious choice of a KR matrix factorization to associate to a web. However, we show 
that the KR-matrix factorizations for all these KR-webs are homotopy equivalent and 
that between two of them there is a canonical choice of homotopy equivalence 
in a certain sense. This allows us to associate an equivalence class of KR-matrix 
factorizations to each ordinary web. After that it is relatively straightforward to 
show the equivalence between the foam and the KR $sl_3$ link homologies.

In Section~\ref{sec:foam} we review the category $\foam$ and the main results 
of \cite{mackaay-vaz}. In Section~\ref{sec:KR-3} we recall some basic facts about 
matrix factorizations and define the universal KR homology for $n=3$. 
Section~\ref{sec:sl3-mf} is the core of the paper. 
In this section we show how to associate equivalence classes of 
matrix factorizations to ordinary webs and use them to construct a 
link homology that is equivalent to Khovanov and Rozansky's. 
In Section~\ref{sec:iso} we establish the equivalence between the foam 
$sl_3$ link homology and the one presented in Section~\ref{sec:sl3-mf}.

We assume familiarity with the papers~\cite{KR, mackaay-vaz}.

\section{The category $\foam$ revisited}\label{sec:foam}

This section contains a brief review of the universal rational $sl_3$ link homology 
using foams as constructed by the authors~\cite{mackaay-vaz} 
following Khovanov's ideas in \cite{khovanovsl3}. 
Here we simply state the basics and the modifications that are necessary 
to relate it to Khovanov and Rozansky's universal $sl_3$ link homology using matrix 
factorizations. We refer to~\cite{mackaay-vaz} for details. 

The category $\foam$ has webs as objects and $\bQ[a,b,c]$-linear combinations of 
foams as morphisms divided by the set of relations $\ell=(3D,CN,S,\Theta)$  
and the \emph{closure relation}, which we all explain below. Note that 
we are using a different normalization of the coefficients\footnote{We thank Scott Morrison 
for spotting a mistake in the coefficients in a previous version of this paper.} in our relations compared to 
\cite{mackaay-vaz}. These are necessary to 
establish the connection with the KR link homology later on. 

$$
\figins{-4.5}{0.2}{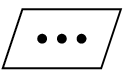}=
a
\figins{-4.5}{0.2}{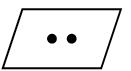}+
b
\figins{-4.5}{0.2}{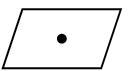}+
c
\figins{-4.5}{0.2}{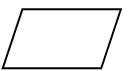}
\rlap{\hspace{18.5ex}\text{(3D)}}
$$
$$
\figwhins{-17}{0.5}{0.23}{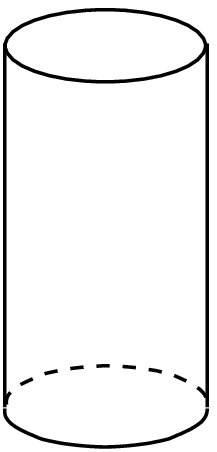}=
4\left(
-\figwhins{-17}{0.5}{0.23}{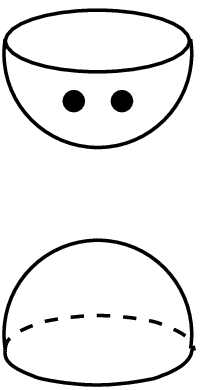}
-\figwhins{-17}{0.5}{0.23}{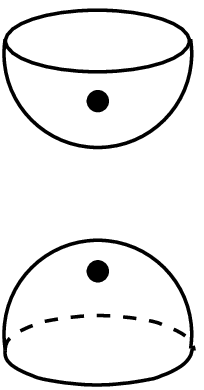}
-\figwhins{-17}{0.5}{0.23}{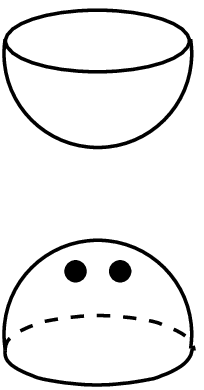}
+a
\left( 
\figwhins{-17}{0.5}{0.23}{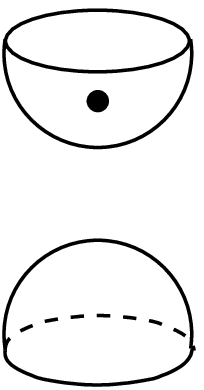}+
\figwhins{-17}{0.5}{0.23}{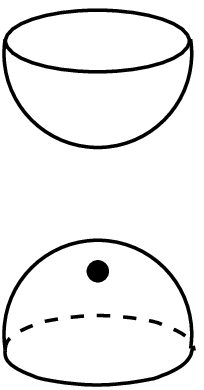}
\right)
+b
\figwhins{-17}{0.5}{0.23}{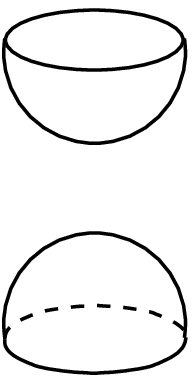}
\right)
\rlap{\hspace{8ex}\text{(CN)}}
$$
$$
\figwins{-8}{0.3}{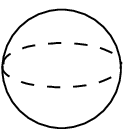}=
\figwins{-8}{0.3}{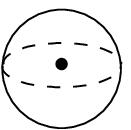}=0,\quad
\figwins{-8}{0.3}{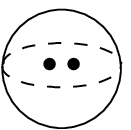}=-\frac{1}{4}
\rlap{\hspace{20.5ex} \text{(S)}}
$$

Let $\theta(\alpha,\beta,\delta)$ denote the theta foam in figure~\ref{fig:theta},
\begin{figure}[ht!]
\labellist
\small\hair 2pt
\pinlabel $\alpha$ at 3 33
\pinlabel $\beta$ at -3 17
\pinlabel $\delta$ at 3 5
\endlabellist
\centering
\figs{0.9}{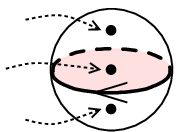}
\caption{A theta-foam}
\label{fig:theta}
\end{figure}
where $\alpha$, $\beta$ and $\delta$ are the number of dots on each facet.
For $\alpha, \beta, \delta\leq 2$ we have
\begin{equation*}
\theta(\alpha,\beta,\delta)=\begin{cases}
\ \ \frac{1}{8} & (\alpha,\beta,\delta)=(1,2,0)\mbox{ or a cyclic permutation} \\ 
- \frac{1}{8} & (\alpha,\beta,\delta)=(2,1,0)\mbox{ or a cyclic permutation} \\ 
\ \ 0 & \mbox{else}
\end{cases}
\rlap{\hspace{6ex}($\Theta$)}
\end{equation*}

The \emph{closure relation} says that any 
$\bQ[a,b,c]$-linear combination of foams, all of which have the same boundary,  
is equal to zero if and only if any common way of closing these foams yields a 
$\bQ[a,b,c]$-linear combination of closed foams whose evaluation is zero.

The category $\foam$ is additive and graded. The $q$-grading in 
$\bQ[a,b,c]$ is defined as 
$$q(1)=0,\quad q(a)=2,\quad q(b)=4,\quad q(c)=6$$
and the degree of a foam $f$ with $|\bullet|$ dots is given by
$$q(f)=-2\chi(f)+\chi(\partial f)+2|\bullet|,$$
where $\chi$ denotes the Euler characteristic.

Using the relations $\ell$ one can prove the identities (\emph{RD}), (\emph{DR}) and 
(\emph{CN}) and Lemma~\ref{lem:KhK} below (for detailed proofs see \cite{mackaay-vaz}).

$$
\figwhins{-16}{0.5}{0.23}{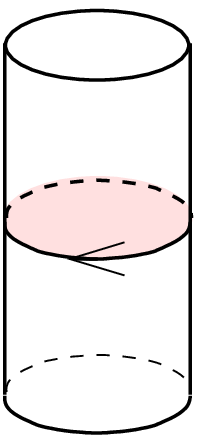}=
2\left(
\figwhins{-16}{0.5}{0.23}{cnecka1}-
\figwhins{-16}{0.5}{0.23}{cnecka2}
\right)
 \rlap{\hspace{22.5ex} \text{(RD)}}
$$
$$
\figins{-20}{0.6}{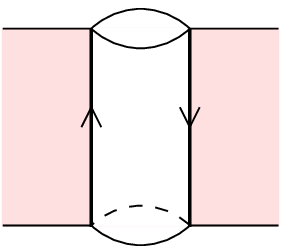}=
2\left(
\figins{-26}{0.75}{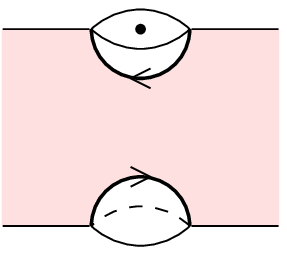}-
\figins{-26}{0.75}{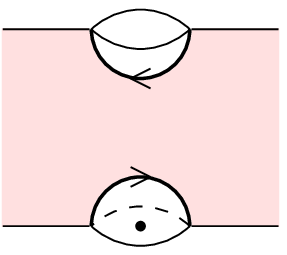}
\right)
\rlap{\hspace{12.8ex}\text{(DR)}}
$$
$$
\figins{-28}{0.8}{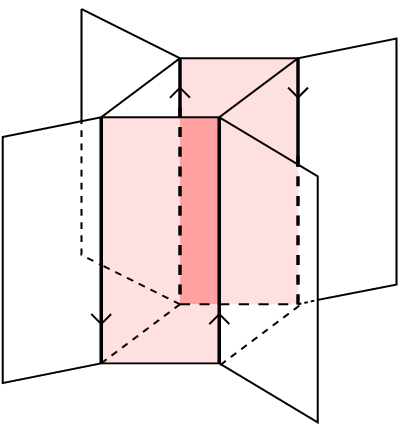}=
-\ \figins{-28}{0.8}{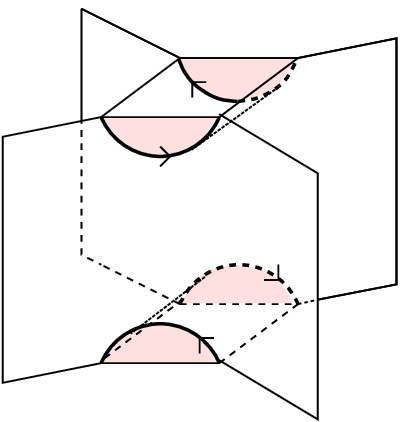}
-\figins{-28}{0.8}{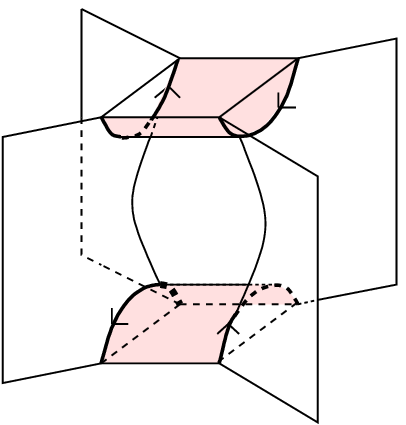}
\rlap{\hspace{12.4ex}\text{(SqR)}}
$$

\begin{lem}\label{lem:KhK}
{\em (Khovanov-Kuperberg relations~\cite{khovanovsl3, Kup})} We have the 
following decompositions in $\foam$:
$$
\figins{-11.5}{0.4}{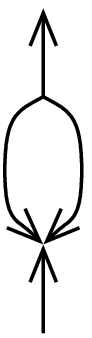} \cong 
\figins{-11.5}{0.4}{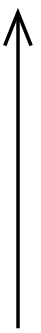}\{-1\}\oplus 
\figins{-11.5}{0.4}{arc-u}\{1\}
\rlap{\hspace{19ex}\text{(Digon Removal)}}
$$
$$
\figins{-10}{0.35}{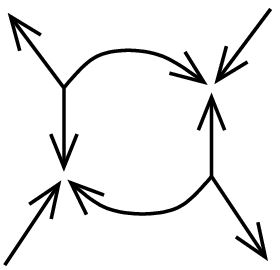}\cong 
\figins{-10}{0.35}{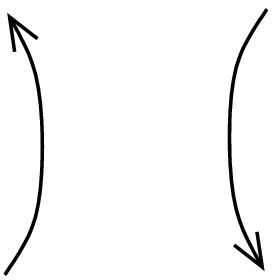}\oplus
\figins{-10}{0.35}{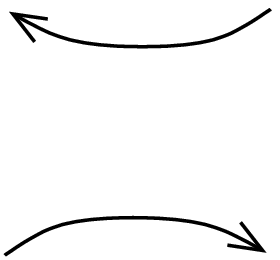}
\rlap{\hspace{16.2ex}\text{(Square Removal)}}
$$
where $\{j\}$ denotes a positive shift in the $q$-grading by $j$.
\end{lem}

The construction of the topological complex from a link diagram is well known by now 
and uses the elementary foams in Figure~\ref{fig:elemfoams}, which we call the 'zip' 
and the 'unzip', to build the differential. 
We follow the conventions in~\cite{mackaay-vaz} and read foams from bottom to top 
when interpreted as morphisms.
\begin{figure}[ht!]
$$
\figins{-18}{0.5}{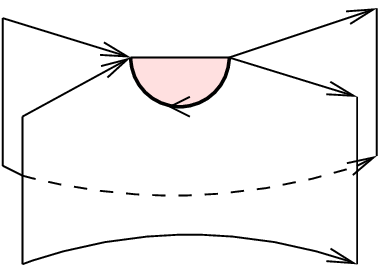}
\quad\qquad
\figins{-18}{0.5}{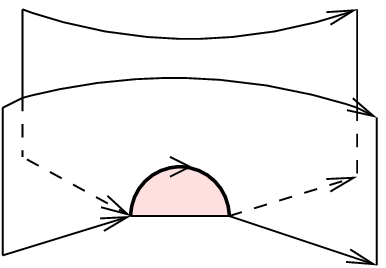}
$$
\caption{Elementary foams}
\label{fig:elemfoams}
\end{figure}

The tautological functor $C$ from $\foam$ to the category 
$\Modgr$ of graded $\bQ[a,b,c]$-modu\-les maps a closed web 
$\Gamma$ to $C(\Gamma)=\Hom_{\foam}\left(\emptyset,\Gamma\right)$ and, for a 
foam $f$ between two closed webs $\Gamma$ and $\Gamma'$, the $\bQ[a,b,c]$-linear map 
$C(f)$ from $C(\Gamma)$ to $C(\Gamma')$ is the one given by composition. 
The $\bQ[a,b,c]$-module $C(\Gamma)$ is graded and the degree of $C(f)$ is equal to 
$q(f)$.

Denote by $\Link$ the category of oriented links in $S^3$ and ambient isotopy classes 
of oriented link cobordisms properly embedded in $S^3\times [0,1]$ and by $\Modbg$ 
the category of bigraded $\bQ[a,b,c]$-modules. The functor $C$ extends to the 
category $\kom(\foam)$ of chain complexes in $\foam$ and the composite with  
the homology functor defines a projective functor $U_{a,b,c}\colon\Link\to\Modbg$.

The theory described above is equivalent to the one in~\cite{mackaay-vaz} after 
tensoring the latter with $\bQ$.

\section{Deformations of the Khovanov-Rozansky $sl_3$-link homology}\label{sec:KR-3}

\subsection{Review of matrix factorizations}
This subsection contains a brief review of matrix factorizations and the properties 
that will be used throughout this paper. We assume familiarity 
with~\cite{KR}. All the matrix factorizations in this paper are 
$\bZ/2\bZ\times\bZ$-graded. Let $R$ be a polynomial ring over $\bQ$. We take the 
degree of each polynomial to be twice its total degree. This way $R$ is 
$\bZ$-graded. Let $W$ be a homogeneous 
element of $R$ of degree $2m$. A matrix factorization of $W$ over $R$ is a 
$\bZ/2\bZ$-graded free $R$-module $M=M_0\oplus M_1$ with $R$-homomorphisms of 
degree $m$
$$M_0\xra{d_0}M_1\xra{d_1}M_0$$
such that $d_1d_0=W\id_{M_0}$ and $d_0d_1=W\id_{M_1}$. 
The $\bZ$-grading of $R$ induces 
a $\bZ$-grading on $M$. The shift functor $\{k\}$ acts on $M$ as
$$M\{k\}=M_0\{k\}\xra{d_0}M_1\{k\}\xra{d_1}M_0\{k\}.$$

A homomorphism $f\colon M\to M'$ of matrix factorizations of $W$ is a pair of 
maps of the same degree 
$f_i\colon M_i\to M'_i$ ($i=0,1$) such that the diagram
$$\xymatrix{
M_0 \ar[r]^{d_0}\ar[d]_{f_0} & M_1\ar[r]^{d_1}\ar[d]_{f_1} & M_0\ar[d]_{f_0} \\
M'_0 \ar[r]^{d'_0} & M'_1\ar[r]^{d'_1} & M'_0
}$$
commutes. It is an isomorphism of matrix factorizations if $f_0$ and $f_1$ are isomorphisms of the underlying modules.
Denote the set of homomorphisms of matrix factorizations from $M$ to $M'$ by 
$$\Hom_{\mf{}}\left(M,M'\right).$$
It has an $R$-module structure with the action of $R$ given by 
$r(f_0,f_1)=(rf_0,rf_1)$ for $r\in R$.
Matrix factorizations over $R$ with homogeneous potential $W$ and homomorphisms 
of matrix factorizations form a graded additive category, which we denote by 
$\mf{R}(W)$. If $W=0$ we simply write $\mf{R}$.

The free $R$-module $\Hom_R\left(M,M'\right)$ of graded $R$-module homomorphisms 
from $M$ to $M'$ is a 2-complex 
$$\xymatrix{
\Hom_R^0\left(M,M'\right)\ar[r]^D &
\Hom_R^1\left(M,M'\right)\ar[r]^D &
\Hom_R^0\left(M,M'\right)
}$$ where
\begin{eqnarray*}
\Hom_R^0\left(M,M'\right) &=& \Hom_R\left(M_0,M'_0\right)\oplus
\Hom_R\left(M_1,M'_1\right) \\
\Hom_R^1\left(M,M'\right) &=& \Hom_R\left(M_0,M'_1\right)\oplus
\Hom_R\left(M_1,M'_0\right)
\end{eqnarray*}
and for $f$ in $\Hom_R^i\left(M,M'\right)$ the differential acts as
$$Df=d_{M'}f-(-1)^ifd_{M}.$$
We define 
$$\Ext\left(M,M'\right)=
\Ext^0\left(M,M'\right)\oplus\Ext^1\left(M,M'\right)
=\Ker{D}/\Image{D},$$
and write $\Ext_{\{m\}}(M,M')$ for the elements of $\Ext(M,M')$ with $\bZ$-degree $m$.

Note that for $f\in\Hom_{\mf{}}\left(M,M'\right)$ we have $Df=0$. We say that two homomorphisms $f$, $g$ in $\Hom_{\mf{}}\left(M,M'\right)$ are homotopic if there is an element $h$ in $\Hom_R^1\left(M,M'\right)$ such that $f-g=Dh$.

Denote by $\Hom_{\hmf{}}\left(M,M'\right)$ the $R$-module of homotopy classes of homomorphisms of matrix factorizations from $M$ to $M'$ and by $\hmf{R}(W)$ the homotopy category of $\mf{R}(W)$.

We have
\begin{eqnarray*}
\Ext^0(M,M') &\cong & \Hom_{\hmf{}}(M,M') \\
\Ext^1(M,M') &\cong & \Hom_{\hmf{}}(M,M'\brak{1})
\end{eqnarray*}

We denote by $M\brak{1}$ and $M_\bullet$ the factorizations
$$M_1\xra{-d_1}M_0\xra{-d_0}M_1$$
and
$$
\left(M_0\right)^*\xra{-\left(d_1\right)^*}
\left(M_1\right)^*\xra{\left(d_0\right)^*}
\left(M_0\right)^*
$$ 
respectively. Factorization $M\brak{1}$ has potential $W$ while factorization $M_\bullet$ has potential $-W$. 

The tensor product $M\otimes_R M_\bullet$ has potential zero and is therefore a 
2-complex. Denoting by $\hy_{\mf{}}$ the homology of matrix factorizations with 
potential zero we have
$$\Ext\left(M,M'\right)\cong \hy_{\mf{}}\left(M'\otimes_R M_\bullet\right)$$
and, if $M$ is a matrix factorization with  $W=0$, 
$$\Ext(R ,M)\cong \hy_{\mf{}}(M).$$
Both these isomorphisms are up to a global shift in the $\bZ$-grading.

\subsubsection*{Koszul Factorizations}
For $a$, $b\in R$ an \emph{elementary Koszul factorization} $\{a,b\}$ over $R$ with potential $ab$ is a factorization of the form
$$R\xra{a}R\left\{{\scriptstyle\frac{1}{2}}\left(\deg_\bZ b -\deg_\bZ a \right)\right\}\xra{b}R.$$
When we need to emphasize the ring $R$ we write this factorization as $\{a,b\}_R$.
It is well known that the tensor product of matrix factorizations $M_i$ with potentials $W_i$ is a matrix factorization with potential $\sum_iW_i$. We restrict to the case where all the $W_i$ are homogeneous of the same degree. Throughout this paper we use tensor products of elementary Koszul factorizations $\{a_j,b_j\}$ to build bigger matrix factorizations, which we write in the form of a \emph{Koszul matrix} as
$$\left\{
\begin{matrix}
a_1 \ , & b_1 \\
 \vdots & \vdots \\
a_k \ , & b_k
\end{matrix}
\right\}$$
We denote by $\{\mathbf{a},\mathbf{b}\}$ the Koszul matrix which has columns 
$(a_1,\ldots,a_k)$ and $(b_1,\ldots ,b_k)$.

Note that the action of the shift $\brak{1}$ on $\{\mathbf{a},\mathbf{b}\}$ 
is equivalent to switching terms in one line of $\{\mathbf{a},\mathbf{b}\}$:
$$\{\mathbf{a},\mathbf{b}\}\brak{1}=\left\{ 
\begin{matrix}
 \vdots\  & \vdots \\
a_{i-1}\ ,& b_{i-1} \\
  -b_i\ , &  -a_i   \\
a_{i+1}\ ,& b_{i+1} \\
 \vdots\  & \vdots
\end{matrix}
\right\}
\brak{{\scriptstyle\frac{1}{2}}\left(\deg_\bZ b_i -\deg_\bZ a_i \right)}.$$
If we choose a different row to switch terms we get a factorization which is 
isomorphic to this one.

We also have that
$$\{\mathbf{a},\mathbf{b}\}_\bullet
\cong
\{\mathbf{a},-\mathbf{b}\}\brak{k}\left\{s_k\right\}
,$$
where
$$s_k=\sum_{i=1}^{k}\deg_\bZ a_i -\frac{k}{2}\deg_\bZ W.$$

Let $R=\bQ[x_1,\ldots,x_k]$ and $R'=\bQ[x_2,\ldots,x_k]$. Suppose that 
$W=\sum_i a_ib_i\in R'$ and $x_1-b_i\in R'$, for a certain $1\leq i\leq k$. 
Let $\{\mathbf{\hat a}^i,\mathbf{\hat b}^i\}$ be the matrix factorization obtained 
from 
$\{\mathbf{a},\mathbf{b}\}$ by deleting the $i$-th row and substituting $x_1$ 
by $x_1-b_1$.   
\begin{lem}[excluding variables]
\label{lem:exvar}
The matrix factorizations $\{\mathbf{a},\mathbf{b}\}$ and 
$\{\mathbf{\hat a}^i,\mathbf{\hat b}^i\}$ are homotopy equivalent. 
\end{lem}
\noindent In \cite{KR} one can find the proof of this lemma and its generalization 
with several variables. 

The following lemma contains three particular cases of proposition 3 in~\cite{KR} 
(see also \cite{KR2}):
\begin{lem}[Row operations]
We have the following isomorphisms of matrix factorizations 
$$\left\{
\begin{matrix}
a_i \ , & b_i \\
a_j \ , & b_j \\
\end{matrix}
\right\} 
\stackrel{[i,j]_\lambda}{\cong}
\left\{
\begin{matrix}
a_i-\lambda a_j \ , & b_i \\
a_j \ , & b_j+\lambda b_i
\end{matrix}
\right\},\qquad\left\{
\begin{matrix}
a_i \ , & b_i \\
a_j \ , & b_j \\
\end{matrix}
\right\} 
\stackrel{[i,j]'_\lambda}{\cong}
\left\{
\begin{matrix}
a_i+\lambda b_j \ , & b_i \\
a_j-\lambda b_i \ , & b_j
\end{matrix}
\right\}$$ 
for $\lambda\in R$. If $\lambda$ is invertible in $R$, we also have
$$\left\{ a_i\ ,\ b_j \right\}\stackrel{[i]_\lambda}{\cong}
\left\{\lambda a_i\ ,\ \lambda^{-1}b_i \right\}.$$
\end{lem}

Recall that a sequence $(a_1,a_2,\ldots ,a_k)$ is called {\em regular} in $R$ 
if $a_j$ is not a zero divisor in $R/_{(a_1,a_2,\ldots , a_{j-1})R}$, for 
$j=1,\ldots,k$.
The proof of the following lemma can be found in \cite{KR2}.
\begin{lem}\label{lem:regseq-iso}
Let $\mathbf{b}=(b_1,b_2, \ldots ,b_k)$, $\mathbf{a}=(a_1,a_2, \ldots ,a_k)$ and $\mathbf{a'}=(a'_1,a'_2, \ldots ,a'_k)$ be sequences in $R$.
If $\mathbf{b}$ is regular and $\sum_ia_ib_i=\sum_ia'_ib_i$ then the factorizations 
$$
\{\mathbf{a}\ , \mathbf{b} \}\ 
\text{ and }\
\{\mathbf{a'}\ , \mathbf{b}\}
$$ are isomorphic.
\end{lem}

A factorization $M$ with potential $W$ is said to be \emph{contractible} if it is isomorphic to a direct sum of 
$$R\xra{1}R\xra{W}\{-{\scriptstyle\frac{1}{2}}\deg_\bZ W \}R\quad 
\text{and}\quad
R\xra{W}R\{{\scriptstyle\frac{1}{2}}\deg_\bZ W   \}\xra{1}R.$$

\subsection{Khovanov-Rozansky homology}\label{ssec:KR-abc}

\begin{defn}
A \emph{KR-web} is a trivalent graph with two types of edges, oriented edges and 
unoriented thick edges, such that each oriented edge has at least one mark. We allow 
open webs which have oriented edges with only one endpoint glued to the rest of the 
graph. Every thick edge has exactly two oriented edges entering one endpoint and 
two leaving the other. 
\end{defn} 

Let $\{1,2,\ldots ,k\}$ be the set of marks in a KR-web $\Gamma$ and $\mathbf{x}$ denote the set 
$\{x_1,x_2,\ldots ,x_k\}$. Denote by $R$ the polynomial ring $\bQ[a,b,c,\mathbf{x}]$, where $a$, $b$, $c$ are formal parameters. We define a $q$-grading in $R$ declaring that 
$$q(1)=0,\quad q(x_i)=2,\,\,\,\text{for all}\,\, i,\quad q(a)=2,\quad q(b)=4,
\quad q(c)=6.$$

The universal rational Khovanov-Rozansky theory for $N=3$ is related to the 
polynomial 
$$p(x)=x^4 -\frac{4a}{3} x^3 - 2b x^2 - 4c x.$$

As in \cite{KR2} we denote by $\hat\Gamma$ the matrix factorization associated to 
a KR-web $\Gamma$. To an oriented arc with marks $x$ and $y$, as in 
Figure~\ref{fig:arc-r}, 
\begin{figure}[ht!]
\labellist
\small\hair 2pt
\pinlabel $y$ at -15 5
\pinlabel $x$ at 112 6
\endlabellist
\centering
\figs{0.5}{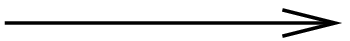}
\caption{An oriented arc}
\label{fig:arc-r}
\end{figure}
we assign the potential
$$W=p(x)-p(y)=x^4-y^4 -\frac{4a}{3} \left(x^3-y^3 \right) - 2b\left(x^2-y^2\right) - 
4c\left(x-y\right)
$$
and the \emph{arc factorization} $\hatarc$, which is given by the Koszul factorization
$$\hatarc=\{\pi_{xy}\ ,\ x-y\}=
\left(
R\xra{\pi_{xy}}R\{-2\}\xra{x-y}R
\right)$$
where
$$\pi_{xy}=\frac{x^4-y^4}{x-y}-\frac{4a}{3}\frac{x^3-y^3}{x-y}-2b\frac{x^2-y^2}{x-y}-4c.$$

\begin{figure}[ht!]
\labellist
\small\hair 2pt
\pinlabel $x_i$ at -8 118
\pinlabel $x_j$ at 52 117
\pinlabel $x_k$ at -8 -9
\pinlabel $x_l$ at 52 -9
\endlabellist
\centering
\figs{0.4}{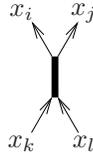}
\caption{A thick edge}
\label{fig:fatedge}
\end{figure}
To the thick edge in Figure~\ref{fig:fatedge} we associate the 
potential $W=p(x_i)+p(x_j)-p(x_k)-p(x_l)$
and the \emph{dumbell factorization} $\hatFatEdge$, which is defined as 
the tensor product of the factorizations
$$
R\{-1\} \xra{u_{ijkl}} R\{-3\}\xra{x_i+x_j-x_k-x_l} R\{-1\}
$$
and
$$
R \xra{v_{ijkl}}  R\xra{x_ix_j-x_kx_l} R
$$
where
\begin{eqnarray*}
u_{ijkl} &=& \frac{(x_i+x_j)^4-(x_k+x_l)^4}{x_i+x_j-x_k-x_l}-(2b+4x_ix_j)(x_i+x_j+x_k+x_l) \\
  & & -\frac{4a}{3}\left(\frac{(x_i+x_j)^3-(x_k+x_l)^3}{x_i+x_j-x_k-x_l}-3x_ix_j\right)-4c,\\
v_{ijkl} &=& 2(x_ix_j+x_kx_l)-4(x_k+x_l)^2 + 4a (x_k+x_l) + 4b.
\end{eqnarray*}
We can write the dumbell factorization as the Koszul matrix 
$$\hatFatEdge=
\left\{
\begin{matrix}
u_{ijkl} \ , & x_i+x_j-x_k-x_l \\
v_{ijkl} \ , & x_ix_j-x_kx_l \\
\end{matrix}
\right\} \{-1\}.$$

The matrix factorization $\hat\Gamma$ of a general KR-web $\Gamma$ composed by 
$E$ arcs and $T$ thick edges is built from the arc and the dumbell 
factorizations as
$$\hat\Gamma=\bigotimes\limits_{e\in E}\hatarc_e\otimes\bigotimes
\limits_{t\in T}\hatFatEdge_t,$$
which is a matrix factorization with potential $W=\sum\epsilon_i p(x_{i})$ 
where $i$ runs over all free ends. 
By convention $\epsilon_i=1$ if the corresponding arc is 
oriented outward and $\epsilon_i=-1$ in the opposite case.
  
\subsubsection{Maps $\chi_0$ and $\chi_1$}\label{ssec:diffs}
Let $\hatorsmooth$ and $\hatFatEdge$ denote the factorizations in 
Figure~\ref{fig:maps-chi}.
\begin{figure}[ht!]
\labellist
\small\hair 2pt
\pinlabel $x_i$ at -8 118
\pinlabel $x_j$ at 52 117
\pinlabel $x_k$ at -8 -10
\pinlabel $x_l$ at 52 -10
\pinlabel $x_i$ at 136 118
\pinlabel $x_j$ at 192 117
\pinlabel $x_k$ at 136 -10
\pinlabel $x_l$ at 192 -10
\pinlabel $\chi_0$ at 94 80
\pinlabel $\chi_1$ at 94 28
\endlabellist
\centering
\figs{0.4}{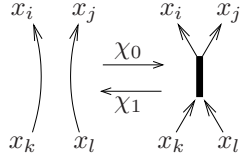}
\caption{Maps $\chi_0$ and $\chi_1$}
\label{fig:maps-chi}
\end{figure}

\noindent The factorization $\hatorsmooth$ is given by 
$$
\begin{pmatrix}R \\ R\{-4\} \end{pmatrix}
\stackrel{P_0}{\lra}
\begin{pmatrix}R\{-2\} \\ R\{-2\} \end{pmatrix}
\stackrel{P_1}{\lra}
\begin{pmatrix}R \\ R\{-4\} \end{pmatrix}
$$
with
$$
P_0=\begin{pmatrix}\pi_{ik} & x_j-x_l \\ \pi_{jl} & -x_i+x_k\end{pmatrix},\qquad
P_1=\begin{pmatrix}x_i-x_k & x_j-x_l \\ \pi_{jl} & -\pi_{ik}\end{pmatrix}.
$$
The factorization $\hatFatEdge$ is given by 
$$
\begin{pmatrix}R\{-1\} \\ R\{-3\} \end{pmatrix}
\stackrel{Q_0}{\lra}
\begin{pmatrix}R\{-3\} \\ R\{-1\} \end{pmatrix}
\stackrel{Q_1}{\lra}
\begin{pmatrix}R\{-1\} \\ R\{-3\} \end{pmatrix}
$$
with
$$
Q_0=\begin{pmatrix}u_{ijkl} & x_ix_j-x_kx_l \\ v_{ijkl} & -x_i-x_j+x_k+x_l
\end{pmatrix},\qquad
Q_1=\begin{pmatrix}x_i+x_j-x_k-x_l & x_ix_j-x_kx_l \\ v_{ijkl} & -u_{ijkl}
\end{pmatrix}.
$$
The maps $\chi_0$ and $\chi_1$ can be described by the pairs of matrices:
$$
\chi_0 = \left(
2\begin{pmatrix}-x_k+x_j & 0 \\ -\alpha & -1 \end{pmatrix},
2\begin{pmatrix}-x_k & x_j \\ 1 & -1 \end{pmatrix}
\right)\quad \text{and}\quad
\chi_1 = \left(
\begin{pmatrix} 1 & 0 \\ -\alpha & x_k-x_j \end{pmatrix},
\begin{pmatrix} 1 & x_j \\ 1 & x_k \end{pmatrix}
\right)
$$
where 
$$\alpha=-v_{ijkl}+\frac{u_{ijkl}+x_iv_{ijkl}-\pi_{jl}}{x_i-x_k}.$$

\noindent The maps $\chi_0$ and $\chi_1$ have degree 1. A straightforward calculation 
shows that $\chi_0$ and $\chi_1$ are homomorphisms of matrix factorizations, and that
$$\chi_0\chi_1=m(2x_j-2x_k)\id(\hatFatEdge)
\qquad
\chi_1\chi_0=m(2x_j-2x_k)\id(\hatorsmooth),
$$
where $m(x_*)$ is multiplication by $x_*$.
Note that we are using a sign convention that is different from the one 
in~\cite{KR}. 
This choice was made to match the signs in the \emph{Digon Removal} relation 
in~\cite{mackaay-vaz}. Note also that the map $\chi_0$ is twice its 
analogue in~\cite{KR}. This way we obtain a theory that is equivalent to the 
one in~\cite{KR} and consistent with our choice of normalization in 
Section~\ref{sec:foam}.

There is another description of the maps $\chi_0$ and $\chi_1$ when the 
webs in Figure~\ref{fig:maps-chi} are closed. In this case both 
$\hatorsmooth$ and $\hatFatEdge$ have potential zero. Acting with a row 
operation on $\hatorsmooth$ we get
$$\hatorsmooth \cong \left\{ 
\begin{matrix}
\pi_{ik}, & x_i+x_j-x_k-x_l \\
\pi_{jl}-\pi_{ik}, & x_j-x_l
\end{matrix}
\right\}.$$
Excluding the variable $x_k$ from $\hatorsmooth$ and from $\hatFatEdge$ 
yields
$$\hatorsmooth\cong\left\{ \pi_{jl}-\pi_{ik},\ x_j-x_l\right\}_R,\qquad
\hatFatEdge\cong\left\{ v_{ijkl},\ (x_i-x_l)(x_j-x_l)\right\}_R,$$
with $R=\bQ[x_i,x_j,x_k,x_l]/(x_k+x_i+x_j-x_l)$. It is straightforward to 
check that $\chi_0$ and $\chi_1$ correspond to the maps $(-2(x_i-x_l), -2)$ 
and $(1, x_i-x_l)$ respectively. This description will be useful in 
Section~\ref{sec:iso}.

\medskip

For a link $L$, we denote by $\KR_{a,b,c}(L)$ the universal rational 
Khovanov-Rozansky cochain complex, which can be obtained from the data above in 
the same way as in \cite{KR}. Let $\HKR_{a,b,c}(L)$ denote the universal rational 
Khovanov-Rozansky homology. We have
$$\HKR_{a,b,c}(\unknot)\cong
\left(\bQ[x,a,b,c]_{/x^3-ax^2-bx-c}\right)
\brak{1}\{-2\}.$$

\subsection{MOY web moves}

One of the main features of the Khovanov-Rozansky theory is the cate\-gorification of 
the set of MOY web moves, which for $n=3$ are described by the homotopy equivalences 
below.
\begin{lem}\label{lem:KR-moy}
We have the following direct sum decompositions:
\begin{eqnarray}
\figins{-13}{0.5}{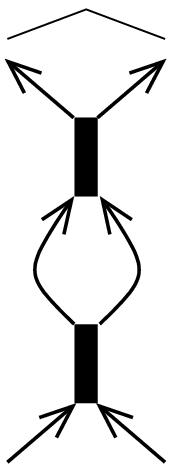}
&\cong&
\figins{-13}{0.5}{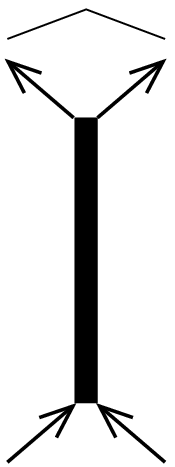}\{-1\}
\oplus
\figins{-13}{0.5}{dr1-KR}\{1\}\ ,
\\
\figins{-13}{0.5}{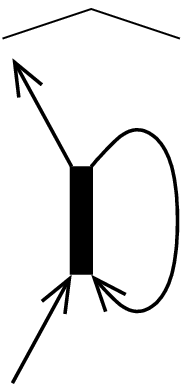}
&\cong&
\figins{-13}{0.5}{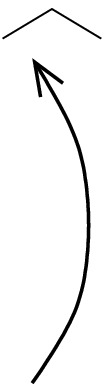}\brak{1}\{-1\}
\oplus
\figins{-13}{0.5}{dr2-KR}\brak{1}\{1\}\ ,
\\
\figins{-12}{0.455}{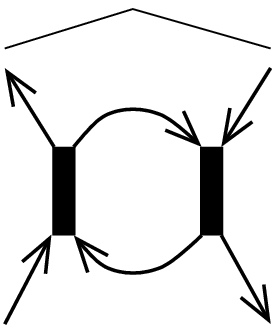}
&\cong&
\figins{-12}{0.45}{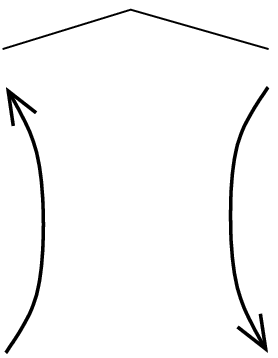}\brak{1}
\oplus
\figins{-12}{0.45}{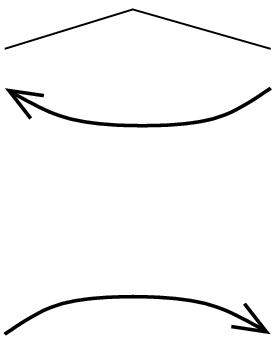}\ ,
\end{eqnarray}
\begin{equation}
\figins{-20}{0.7}{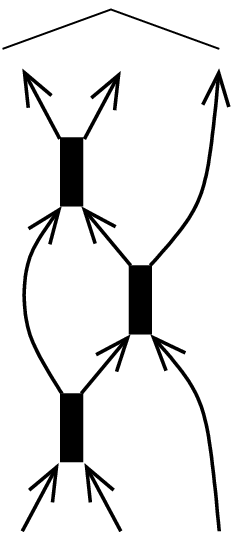}\
\oplus
\figins{-20}{0.7}{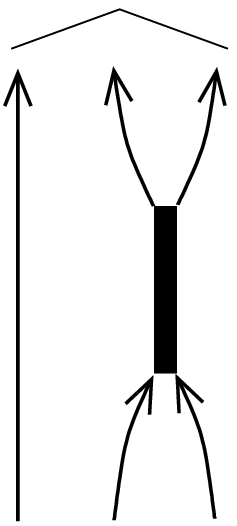}
\cong
\figins{-20}{0.7}{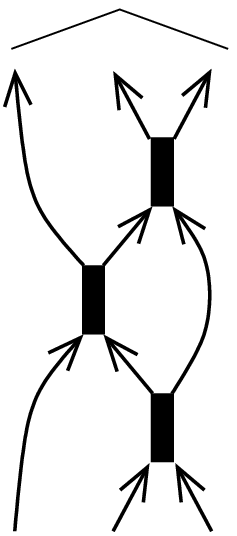}\
\oplus
\figins{-20}{0.7}{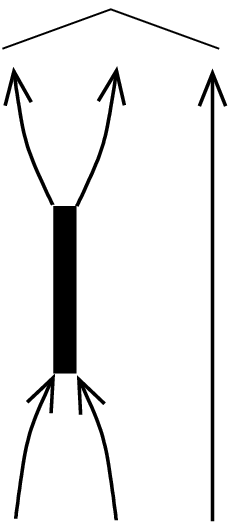}\ .
\end{equation}
\end{lem}

The last relation is a consequence of two relations involving a triple edge
$$
\figins{-20}{0.7}{square2a-KR}
\cong 
\figins{-20}{0.7}{sqr2bb-KR}\oplus
\figins{-20}{0.7}{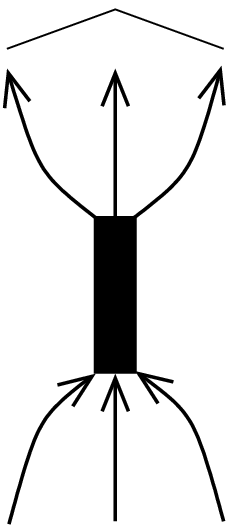}
\qquad , \qquad
\figins{-20}{0.7}{square2b-KR}
\cong 
\figins{-20}{0.7}{sqr2aa-KR}\oplus
\figins{-20}{0.7}{trpl-KR}.
$$ 

\medskip

The factorization assigned to the triple edge in Figure~\ref{fig:trp-KR} has
potential
$$W=p(x_1)+p(x_2)+p(x_3)-p(x_4)-p(x_5)-p(x_6).$$
\begin{figure}[ht!]
\labellist
\small\hair 2pt
\pinlabel $x_1$ at -8 196
\pinlabel $x_2$ at 48 196
\pinlabel $x_3$ at 100 196
\pinlabel $x_4$ at -8 -16
\pinlabel $x_5$ at 48 -16
\pinlabel $x_6$ at 100 -16
\endlabellist
\centering
\figs{0.28}{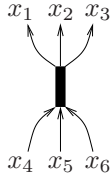}
\caption{Triple edge factorization}
\label{fig:trp-KR}
\end{figure}

Let $h$ be the unique three-variable polynomial such that
$$h(x+y+z,xy+xz+yz,xyz)=p(x)+p(y)+p(z)$$
and let
\begin{gather*}
e_1=x_1+x_2+x_3,
\qquad 
e_2=x_1x_2+x_1x_3+x_2x_3,
\quad
e_3=x_1x_2x_3, \\
s_1=x_4+x_5+x_6,
\qquad 
s_2=x_4x_5+x_4x_6+x_5x_6,
\quad
s_3=x_4x_5x_6.
\end{gather*}
Define
\begin{eqnarray*}
h_1 &=& \frac{h(e_1,e_2,e_3)-h(s_1,e_2,e_3)}{e_1-s_1} \\
h_2 &=& \frac{h(s_1,e_2,e_3)-h(s_1,s_2,e_3)}{e_2-s_2} \\
h_3 &=& \frac{h(s_1,s_2,e_3)-h(s_1,s_2,s_3)}{e_3-s_3}
\end{eqnarray*}
so that we have $W=h_1(e_1-s_1)+h_2(e_2-s_2)+h_3(e_3-s_3)$. The matrix factorization 
$\hat\Upsilon$ corresponding to the triple edge is defined by the Koszul 
matrix
$$\hat\Upsilon=\left\{
\begin{matrix}
h_1\ ,& e_1-s_1 \\
h_2\ ,& e_2-s_2 \\
h_3\ ,& e_3-s_3
\end{matrix}\right\}_R\{-3\},$$
where $R=\bQ[a,b,c,x_1,\ldots ,x_6]$. 
The matrix factorization $\hat\Upsilon$ is the tensor product of 
the matrix factorizations
$$R\xra{h_i}R\{2i-4\}\xra{e_i-s_i}R,\qquad i=1,2,3,$$
shifted down by 3.

\subsection{Cobordisms}\label{ssec:cob-mf}
In this subsection we show which homomorphisms of matrix factorizations we 
associate to the elementary singular KR-cobordisms. 
We do not know if these can be put together in a well-defined way for 
arbitrary singular KR-cobordisms. However, we will prove that that is possible 
for arbitrary ordinary singular cobordisms in Section~\ref{sec:iso}. 

Recall that the elementary KR-cobordisms are the \emph{zip}, the \emph{unzip} 
(see Figure~\ref{fig:maps-chi}) and the elementary cobordisms in 
Figure~\ref{fig:cobs}. 
To the zip and the unzip we associate the maps $\chi_0$ and $\chi_1$, 
respectively, as defined in Subsection~\ref{ssec:diffs}. For each 
elementary cobordism in Figure~\ref{fig:cobs} we define a 
homomorphism of matrix factorizations as below, 
following Khovanov and Rozansky~\cite{KR}.
\begin{figure}[ht!]
\centering
\figins{0}{0.25}{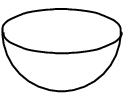}\qquad
\figins{0}{0.25}{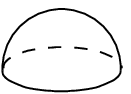}\qquad
\figins{-7}{0.6}{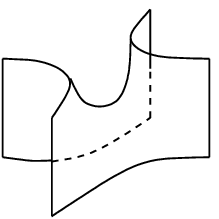}
\caption{Elementary cobordisms}
\label{fig:cobs}
\end{figure}

\noindent The \emph{unit} and the \emph{trace} map 
$$\imath\colon\bQ[a,b,c]\brak{1} \to \hatunknot$$
$$\varepsilon\colon\hatunknot\to \bQ[a,b,c]\brak{1}$$
are the homomorphisms of matrix factorizations induced by the 
maps (denoted by the same symbols)
$$\begin{array}{ll}
\imath\colon\bQ[a,b,c]\to\bQ[a,b,c][X]_{/ X^3-aX^2-bX-c}\{-2\}, &  1\mapsto 1 \\
\\
\varepsilon\colon\bQ[a,b,c][X]_{/ X^3-aX^2-bX-c}\{-2\}\to\bQ[a,b,c], &  X^k\mapsto \begin{cases}
-\frac{1}{4}, & k=2 \\
\ \ 0 , & k<2 \end{cases}
\end{array}$$
using the isomorphisms 
$$\hat\emptyset\cong\bQ[a,b,c]\to 0\to\bQ[a,b,c]$$ 
and
$$\hatunknot\cong 0\to\bQ[a,b,c][X]_{/ X^3-aX^2-bX-c}\{-2\}\to 0.$$

Let $\hattwoedgesop$ and $\hathtwoedgesop$ be the factorizations in 
Figure~\ref{fig:saddle}.
\begin{figure}[ht!]
\labellist
\small\hair 2pt
\pinlabel $x_1$ at -4 89
\pinlabel $x_2$ at 82 89
\pinlabel $x_3$ at -4 -6
\pinlabel $x_4$ at 82 -6
\pinlabel $x_1$ at 170 89
\pinlabel $x_2$ at 262 89
\pinlabel $x_3$ at 170 -6
\pinlabel $x_4$ at 262 -6
\pinlabel $\eta$ at 128 54
\endlabellist
\centering
\figs{0.4}{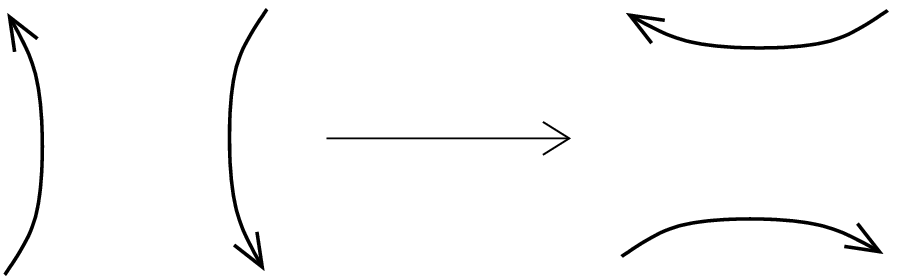}
\caption{Saddle point homomorphism}
\label{fig:saddle}
\end{figure}

The matrix factorization $\hattwoedgesop$ is given by 
$$
\begin{pmatrix}R \\ R\{-4\} \end{pmatrix}
\xra{\begin{pmatrix}
\pi_{13} & x_4-x_2\\
\pi_{24} & x_3-x_1 
\end{pmatrix}}
\begin{pmatrix}R\{-2\} \\ R\{-2\} \end{pmatrix}
\xra{\begin{pmatrix}
x_1-x_3  & x_4-x_2\\
\pi_{24} & -\pi_{13}
\end{pmatrix}}
\begin{pmatrix}R \\R\{-4\} \end{pmatrix}
$$
and 
$\hathtwoedgesop\brak{1}$ is given by 
$$
\begin{pmatrix}R\{-2\} \\ R\{-2\}  \end{pmatrix}
\xra{\begin{pmatrix}
  x_2-x_1  & x_3-x_4 \\
 -\pi_{34} & \pi_{12}
\end{pmatrix}}
\begin{pmatrix}R \\ R\{-4\} \end{pmatrix}
\xra{\begin{pmatrix}
 -\pi_{12} & x_3-x_4\\
 -\pi_{34} & x_1-x_2 
\end{pmatrix}}
\begin{pmatrix}R\{-2\} \\ R\{-2\} \end{pmatrix}
$$

\noindent To the saddle cobordism between the webs $\twoedgesop$ and $\htwoedgesop$ we associate the homomorphism of matrix factorizations $\eta\colon\hattwoedgesop\to \hathtwoedgesop\brak{1}$ described by the pair of matrices
$$\eta_0=
\begin{pmatrix}
 e_{123}+e_{124} & 1 \\
-e_{134}-e_{234} & 1 
\end{pmatrix},
\qquad
\eta_1=
\begin{pmatrix}
 -1 & 1 \\
-e_{123}-e_{234} & -e_{134}-e_{123} 
\end{pmatrix}
$$
where 
\begin{eqnarray*}
e_{ijk} &=&
\frac{(x_k-x_j)p(x_i)+(x_i-x_k)p(x_j)+(x_j-x_i)p(x_k)}{2(x_i-x_j)(x_j-x_k)(x_k-x_i)}
\\
  &=& \frac{1}{2}\left(x_i^2+x_j^2+x_k^2+x_ix_j+x_ix_k+x_jx_k\right)-\frac{2a}{3}\left(x_i+x_j+x_k\right)-b.
\end{eqnarray*}
The homomorphism $\eta$ has degree 2. If $\twoedgesop$ and $\htwoedgesop$ belong to 
open KR-webs the homomorphism $\eta$ is defined only up to a sign (see~\cite{KR}). 
In Section~\ref{sec:sl3-mf} we will show that for ordinary closed webs there is 
no sign ambiguity.

\section{A matrix factorization theory for $sl_3$ webs}\label{sec:sl3-mf}

As mentioned in the introduction, the main problem in comparing the foam and 
the matrix factorization $sl_3$ 
link homologies is that one has to deal with two different sorts of webs. In the 
foam approach, where one uses ordinary trivalent webs, all edges are thin and 
oriented, whereas for the matrix factorizations Khovanov and Rozansky used KR-webs, 
which also contain thick unoriented edges. In general there are various 
KR-webs that one can associate to a given web. Therefore it apparently is not clear 
which KR matrix factorization to associate to a web. However, in Proposition 
\ref{prop:ident-vert} we show that this ambiguity is not problematic. Our proof 
of this result is rather roundabout and requires a matrix 
factorization for each vertex. In this way we associate a matrix factorization 
to each web. For each choice of KR-web associated to a given web the KR-matrix 
factorization is a quotient of ours, obtained by identifying the vertex variables 
pairwise according to the thick edges. We then show that for a given web two such 
quotients are always homotopy equivalent. This is the 
main ingredient which allows us to establish the equivalence between the foam and 
the matrix factorization $sl_3$ link homologies.        

Recall that a \emph{web} is an oriented trivalent graph where near each vertex either 
all edges are oriented into it or away from it. We call the former vertices of 
($-$)-\emph{type} and the latter vertices of ($+$)-\emph{type}. To associate 
matrix factorizations to webs we impose that 
each edge have at least one mark.

\begin{figure}[ht!]
\labellist
\small\hair 2pt
\pinlabel $x$ at -8 115
\pinlabel $y$ at 92 114
\pinlabel $z$ at 42 -10
\pinlabel $v$ at 28 56
\endlabellist
\centering
\figs{0.35}{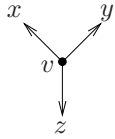}
\caption{A vertex of ($+$)-type}
\label{fig:3vert-out}
\end{figure}
%
\subsection{The 3-vertex}\label{ssec:triv}
Consider the 3-vertex of ($+$)-type in Figure~\ref{fig:3vert-out}, with emanating 
edges marked $x,y,z$. The polynomial
$$p(x)+p(y)+p(z)=x^4+y^4+z^4-\frac{4a}{3}\left(x^3+y^3+z^3 \right)-2b\left(x^2+y^2+z^2\right)-4c\left(x+y+z\right)$$
can be written as a polynomial in the elementary symmetric polynomials
$$p_v(x+y+z,xy+xz+yz,xyz)=p_v(e_1,e_2,e_3).$$

Using the methods of Section~\ref{sec:KR-3} we can obtain a matrix factorization of 
$p_v$, but if we tensor together two of these, then we obtain a matrix factorization 
which is not homotopy equivalent to the dumbell matrix factorization. This can be seen 
quite easily, since the new Koszul matrix has 6 rows and 
only one extra variable. This extra variable can be excluded at the expense of 1 
row, but then we get a Koszul matrix with 5 rows, whereas the dumbell Koszul matrix 
has only 2. To solve this problem we introduce a 
set of three new variables for each vertex\footnote{Khovanov had already observed 
this problem for the undeformed case and suggested to us 
the introduction of one vertex variable in that case.}.
Introduce the \emph{vertex variables} $v_1$, $v_2$, $v_3$ with $q(v_i)=2i$ and define 
the \emph{vertex ring}  
$$R_v=\bQ[a,b,c][x,y,z,v_1,v_2,v_3].$$

\noindent We define the potential as $$W_v=p_v(e_1,e_2,e_3)-p_v(v_1,v_2,v_3).$$
We have
\begin{eqnarray*}
W_v &=& \frac{p_v(e_1,e_2,e_3)-p_v(v_1,e_2,e_3)}{e_1-v_1}\left(e_1-v_1\right) \\
    &+& \frac{p_v(v_1,e_2,e_3)-p_v(v_1,v_2,e_3)}{e_2-v_2}\left(e_2-v_2\right) \\
    &+& \frac{p_v(v_1,v_2,e_3)-p_v(v_1,v_2,v_3)}{e_3-v_3}\left(e_3-v_3\right) \\
    &=& g_1\left(e_1-v_1\right) + g_2\left(e_2-v_2\right) + g_3\left(e_3-v_3\right),
\end{eqnarray*}
where the polynomials $g_i$ ($i=1,2,3$) have the explicit form
\begin{eqnarray}\label{eqn:gfactor1}
g_1 &=& \frac{e_1^4-v_1^4}{e_1-v_1}-4e_2(e_1+v_1)+4e_3-
        \frac{4a}{3}\left(\frac{e_1^3-v_1^3}{e_1-v_1}-3e_2\right) -2b(e_1+v_1)-4c \\
g_2 &=& 2(e_2+v_2)-4v_1^2+4av_1 + 4b \\
g_3 &=& 4(v_1-a).\label{eqn:gfactor3}
\end{eqnarray}
We define the 3-vertex factorization $\hatYGraph_{v_+}$ as the tensor product of the 
factorizations
$$R_v\xra{g_i} R_v\{2i-4\}\xra{e_i-v_i} R_v,
\qquad (i=1,2,3)$$
shifted by $-3/2$ in the $q$-grading and by $1/2$ in the $\bZ/2\bZ$-grading, which 
we write in the form of the Koszul matrix
$$\hatYGraph_{v_+}=
\left\{ 
\begin{matrix}
 g_1\ , & e_1-v_1  \\
 g_2\ , & e_2-v_2  \\
 g_3\ , & e_3-v_3
\end{matrix}
\right\}_{R_v}
\left\{-3/2\right\}\brak{1/2}
.$$

If $\YGraph_v$ is a 3-vertex of $-$-type with incoming edges marked $x,y,z$ 
we define
$$\hatYGraph_{v_-}=
\left\{ 
\begin{matrix}
 g_1\ , & v_1-e_1 \\
 g_2\ , & v_2-e_2 \\
 g_3\ , & v_3-e_3
\end{matrix}
\right\}_{R_v}
\left\{-3/2\right\}\brak{1/2},$$
with $g_1$, $g_2$, $g_3$ as above.

\begin{lem}
We have the following homotopy equivalences in $\End_{\mf{}}{(\hatYGraph_{v_\pm})}$:
$$m(x+y+z)\cong m(a),\quad 
m(xy+xz+yz)\cong m(-b),\quad 
m(xyz)\cong m(c).$$
\end{lem}
\proof
For a matrix factorization $\hat{M}$ over $R$ with potential $W$ the homomorphism 
$$R\ra\End_{\mf{}}(\hat{M}),\quad r\mapsto m(r)$$ 
factors through the Jacobi algebra of $W$ and up to shifts, the Jacobi algebra of 
$W_v$ is
$$J_{W_v}\cong\bQ[a,b,c,x,y,z]_{/ 
\{
x+y+z=a,\ 
xy+xz+yz=-b,\ 
xyz=c
\}}.\rlap{\hspace{0.572in}\qedsymbol}$$ 

\subsection{Vertex composition}
The elementary webs considered so far can be combined to produce bigger webs. Consider a general web $\Gamma_v$ composed by $E$ arcs between marks and $V$ vertices. Denote by $\partial E$ the set of free ends of $\Gamma_v$ and by $(v_{i_1},v_{i_2},v_{i_3})$ the vertex variables corresponding to the vertex $v_i$. We have
\begin{equation*}\label{eq:bigmf}
\hat\Gamma_v =\bigotimes\limits_{e\in E}\hatarc_e\otimes\bigotimes_{v\in V}\hatYGraph_v.
\end{equation*}
Factorization $\hatarc_e$ is the arc factorization introduced in 
Subsection~\ref{ssec:KR-abc}. This is a matrix factorization with potential 
$$W=\sum\limits_{i\in\partial E(\Gamma)}\epsilon_i p(x_{i})+\sum\limits_{v_j\in V(\Gamma)}\epsilon_j^vp_v(v_{j_1},v_{j_2},v_{j_3})=W_\mathbf{x}+W_\mathbf{v}$$ where $\epsilon_i=1$ if the corresponding arc is oriented to it or $\epsilon_i=-1$ in the opposite case and $\epsilon_j^v=1$ if $v_j$ is of positive type and $\epsilon_j^v=-1$ in the opposite case.

From now on we only consider open webs in which the number of free ends oriented inwards equals the number of free ends oriented outwards. This implies that the number of vertices of ($+$)-type equals the number of vertices of ($-$)-type.

Let $R$ and $R_\mathbf{v}$ denote the rings $\bQ[a,b,c][\mathbf{x}]$ and 
$R[\mathbf{v}]$ 
respectively. Given two vertices, $v_i$ and $v_j$, of opposite type, we 
can take the quotient by the ideal generated by $v_i-v_j$. 
The potential becomes independent of $v_i$ and $v_j$, because 
they appeared with opposite signs, and we can 
exclude the common vertex variables corresponding to $v_i$ and $v_j$ as in 
Lemma~\ref{lem:exvar}. This is possible because in all our examples the 
Koszul matrices have linear terms which involve vertex and edge variables.  
The matrix factorization which we obtain in this way we represent graphically 
by a virtual edge, as in 
Figure~\ref{fig:virtualedge}. 
\begin{figure}[h!]
$$
\figins{0}{0.35}{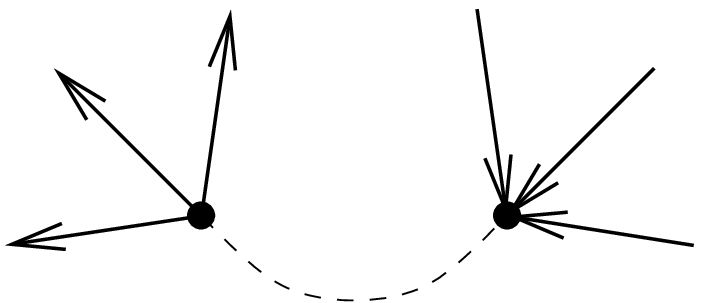}
\qquad\quad
\figins{0}{0.38}{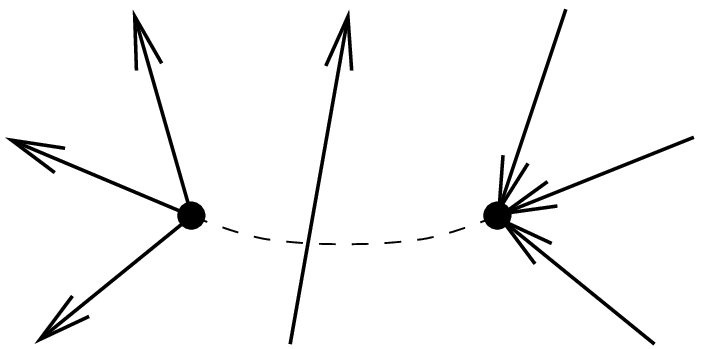}
$$
\caption{Virtual edges}
\label{fig:virtualedge}
\end{figure}
A virtual edge can cross other virtual edges and ordinary edges and does not have 
any mark.

If we pair every positive vertex in $\Gamma_v$ to a negative one, the above 
procedure leads to a {\em complete identification} of the vertices of $\Gamma_v$ 
and a corresponding matrix factorization $\zeta(\hat\Gamma_v)$. 
A different complete identification yields a different matrix factorization 
$\zeta'(\hat\Gamma_v)$.

\begin{prop}\label{prop:ident-vert}
Let $\Gamma_v$ be a closed web. Then $\zeta(\hat\Gamma_v)$ and 
$\zeta'(\hat\Gamma_v)$ are isomorphic, up to a shift in the $\bZ/2\bZ$-grading.
\end{prop}

\noindent To prove Proposition~\ref{prop:ident-vert} we need some technical results.
\begin{lem}\label{lem:KR-3edge}
Consider the web $\Gamma_v$ and the KR-web $\Upsilon$ below.
$$\xymatrix@R=1mm{
\figins{0}{0.7}{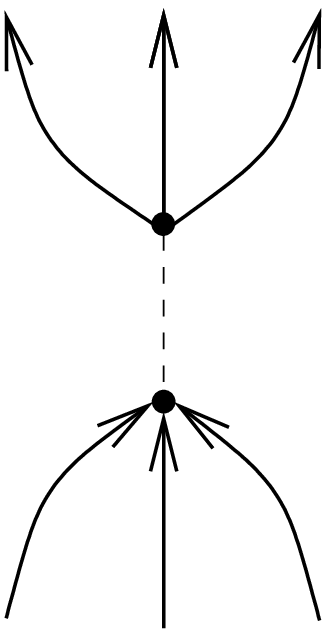} &
\figins{0}{0.7}{trpl_edge} \\
\Gamma_v & \Upsilon
}.$$
Then $\zeta(\hat\Gamma_v\brak{1})$ is the factorization of the triple edge 
$\hat\Upsilon$ of Khovanov-Rozansky for $n=3$.
\end{lem}
\begin{proof} Immediate. 
\end{proof}

\begin{figure}[ht!]
\raisebox{20pt}{$\Gamma_v =$\quad\ }
\labellist
\small\hair 2pt
\pinlabel $x_i$ at -12 130
\pinlabel $x_j$ at 136 129
\pinlabel $x_k$ at -12 -6
\pinlabel $x_l$ at 136 -6
\pinlabel $x_r$ at 36 66 
\endlabellist
\centering
\figs{0.35}{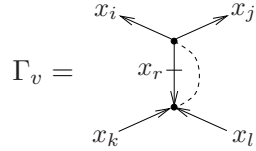}
\caption{A double edge}
\label{fig:thick-KR}
\end{figure}
\begin{lem}\label{lem:fatedge}
Let $\Gamma_v$ be the web in figure~\ref{fig:thick-KR}.
Then $\zeta(\hat\Gamma_v)$ is isomorphic to the factorization assigned to the 
thick edge of Khovanov-Rozansky for $n=3$.
\end{lem}
\begin{proof}
Let $\hat\Gamma_+$ and $\hat\Gamma_-$ be the Koszul factorizations for the upper and
lower vertex in $\Gamma_v$ respectively and $v^\pm_i$ denote the corresponding 
sets of vertex variables.  
We have
$$
\hat\Gamma_+ =
\left\{\begin{matrix}
g^+_1\ ,& x_i+x_j+x_r-v^+_1 \\
g^+_2\ ,& x_ix_j+x_r(x_i+x_j)-v^+_2 \\
g^+_3\ ,& x_ix_jx_r-v^+_3
\end{matrix}\right\}_{R_{v_+}}\{-3/2\}\brak{1/2}$$
and
$$
\hat\Gamma_- =
\left\{\begin{matrix}
g^-_1\ ,& v^-_1 - x_k-x_l-x_r \\
g^-_2\ ,& v^-_2 - x_kx_l-x_r(x_k+x_l) \\
g^-_3\ ,& v^-_3 - x_kx_lx_r
\end{matrix}\right\}_{R_{v_-}}\{-3/2\}\brak{1/2}.$$
The explicit form of the polynomials $g^\pm_i$ is given in 
Equations~(\ref{eqn:gfactor1}-\ref{eqn:gfactor3}).
Taking the tensor product of $\hat\Gamma_+$ and $\hat\Gamma_-$, identifying vertices $v_+$ and $v_-$ and excluding the vertex variables yields
$$
\zeta(\hat\Gamma_v) = 
\left(\hat\Gamma_+\otimes\hat\Gamma_-\right)_{/ \mathbf{v}_+ - \mathbf{v}_-}\cong
\left\{\begin{matrix}
g_1\ ,& x_i+x_j-x_k-x_l \\
g_2\ ,& x_ix_j - x_kx_l + x_r(x_i+x_j-x_k-x_l) \\
g_3\ ,& x_r(x_ix_j-x_kx_l) 
\end{matrix}\right\}_R\{-3\}\brak{1},
$$
where
$$
g_i=
\left. g_i^+ 
\right|_{\{v_1^+=x_k+x_l+x_r,\ v_2^+=x_kx_l+x_r(x_k+x_l),\ v_3^+=x_rx_kx_l\} }.$$

This is a factorization over the ring 
$$R=\bQ[a,b,c][x_i,x_j,x_k,x_l,x_r,\mathbf{v}]_{/I}\cong \bQ[a,b,c][x_i,x_j,x_k,x_l,x_r]$$
where $I$ is the ideal generated by
$$\{
v_1 - x_r - x_k - x_l,\
v_2 - x_kx_l - x_r(x_k+x_l),\
v_3 - x_kx_lx_r
\}.$$
Using $g_3=4(x_r+x_k+x_l-a)$ and acting with the shift functor $\brak{1}$ on the third row one can write
$$
\zeta(\hat\Gamma_v) \cong
\left\{\begin{matrix}
g_1\ ,& x_i+x_j-x_k-x_l \\
g_2\ ,& x_ix_j - x_kx_l + x_r(x_i+x_j-x_k-x_l) \\
-x_r(x_ix_j-x_kx_l)\ ,&  -4(x_r+x_k+x_l-a)
\end{matrix}\right\}_R\{-1\}
$$
which is isomorphic, by a row operation, to the factorization
$$
\left\{\begin{matrix}
g_1 + x_rg_2 \ ,& x_i+x_j-x_k-x_l \\
g_2\ ,& x_ix_j - x_kx_l \\
-x_r(x_ix_j-x_kx_l)\ ,&  -4(x_r+x_k+x_l-a)
\end{matrix}\right\}_R\{-1\}.
$$
Excluding the variable $x_r$ from the third row gives
$$
\zeta(\hat\Gamma_v) \cong
\left\{\begin{matrix}
g_1 + (a-x_k-x_l) g_2\ ,& x_i+x_j-x_k-x_l \\
g_2 \ ,& x_ix_j - x_kx_l
\end{matrix}\right\}_{R'}\{-1\}
$$
where 
$$
R'=\bQ[a,b,c][x_i,x_j,x_l,x_r]_{/x_r-a+x_k+x_l} 
\cong
\bQ[a,b,c][x_i,x_j,x_k,x_l].$$ 
The claim follows from Lemma~\ref{lem:regseq-iso}, since both are factorizations 
over $R'$ with the same potential and the same second column, the terms in which 
form a regular sequence in $R'$. As a matter of fact, using a row operation 
one can write
 $$
\zeta(\hat\Gamma_v) \cong
\left\{\begin{matrix}
g_1 + (a-x_k-x_l) g_2+2(a-x_k-x_l)(x_ix_j-x_kx_l) ,& x_i+x_j-x_k-x_l \\
g_2 + 2(a-x_k-x_l)(x_i+x_j-x_k-x_l) ,& x_ix_j - x_kx_l
\end{matrix}\right\}\{-1\}
$$
and check that the polynomials in the first column are exactly the polynomials $u_{ijkl}$ and $v_{ijkl}$ corresponding to the factorization assigned to the thick edge in Khovanov-Rozansky theory. 
\end{proof}

\begin{lem}\label{lem:swapedges}
Let $\Gamma_v$ be a closed web and $\zeta$ and $\zeta'$ two complete identifications 
that only differ in the region depicted 
in Figure~\ref{fig:swap}, where $T$ is a part of the diagram whose orientation is not 
important. Then there is an isomorphism $\zeta(\hat\Gamma_v)\cong
\zeta'(\hat\Gamma_v)\brak{1}$.
\begin{figure}[ht!]
\bigskip
\labellist
\small\hair 2pt
\pinlabel $x_1$ at -22 136
\pinlabel $x_2$ at -8 194
\pinlabel $x_3$ at 54 212
\pinlabel $x_4$ at 142 212
\pinlabel $x_5$ at 200 194
\pinlabel $x_6$ at 214 136
\pinlabel $x_7$ at 214 56
\pinlabel $x_8$ at 200 0
\pinlabel $x_9$ at 142 -16
\pinlabel $x_{10}$ at 54 -16
\pinlabel $x_{11}$ at -8 0
\pinlabel $x_{12}$ at -24 56
\pinlabel $x_1$ at 310 136
\pinlabel $x_2$ at 320 194
\pinlabel $x_3$ at 372 212
\pinlabel $x_4$ at 470 212
\pinlabel $x_5$ at 528 194
\pinlabel $x_6$ at 542 136
\pinlabel $x_7$ at 542 56
\pinlabel $x_8$ at 528 0
\pinlabel $x_9$ at 470 -16
\pinlabel $x_{10}$ at 372 -16
\pinlabel $x_{11}$ at 320 0
\pinlabel $x_{12}$ at 304 56
\endlabellist
\centering
\figs{0.35}{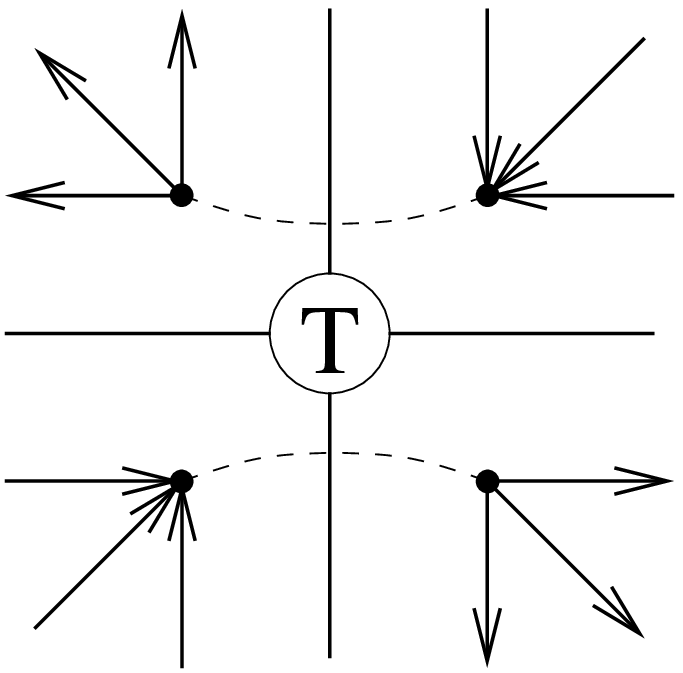}
\qquad\qquad
\figs{0.35}{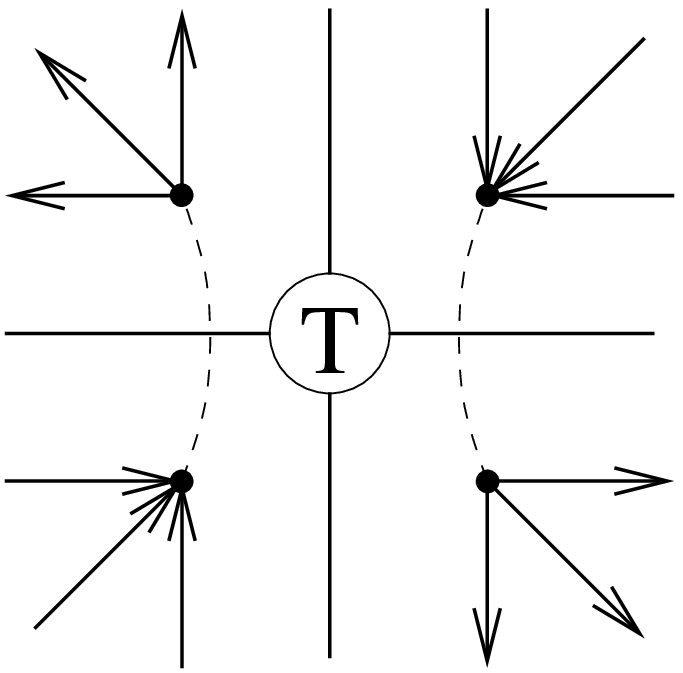}
\put(-154,-30){$\zeta(\hat\Gamma_v)$}\put(-39,-30){$\zeta'(\hat\Gamma_v)$}
\caption{Swapping virtual edges}
\label{fig:swap}
\end{figure}
\end{lem}

\begin{proof}
Denoting by $\hat{M}$ the tensor product of $\hat T$ with the factorization 
corresponding to the part of the diagram not depicted in Figure~\ref{fig:swap} 
we have
$$\zeta(\hat\Gamma_v) \cong \hat M\otimes
\left\{\begin{matrix}
g_1\ , & x_1+x_2+x_3-x_4-x_5-x_6 \\
g_2\ , & x_1x_2+(x_1+x_2)x_3-x_5x_6-x_4(x_5+x_6)\\
g_3\ , & x_1x_2x_3-x_4x_5x_6 \\
g_1'\ , & x_7+x_8+x_9-x_{10}-x_{11}-x_{12} \\
g_2'\ , & x_8x_9+x_7(x_8+x_9)-x_{10}x_{11}-(x_{10}+x_{11})x_{12} \\
g_3'\ , & x_7x_8x_9-x_{10}x_{11}x_{12}
\end{matrix}\right\}\{-6\}$$
with polynomials $g_i$ and $g_i'$ ($i=1,2,3$) given by 
Equations~(\ref{eqn:gfactor1}-\ref{eqn:gfactor3}). Similarly
$$\zeta'(\hat\Gamma_v) \cong \hat M\otimes
\left\{\begin{matrix}
h_1\ , & x_7+x_8+x_9-x_4-x_5-x_6 \\
h_2\ , & x_8x_9+x_7(x_8+x_9)-x_5x_6-x_4(x_5+x_6) \\
h_3\ , & x_7x_8x_9-x_4x_5x_6  \\
h_1'\ , & x_1+x_2+x_3-x_{10}-x_{11}-x_{12} \\
h_2'\ , & x_1x_2+(x_1+x_2)x_3-x_{10}x_{11}-(x_{10}+x_{11})x_{12}\\
h_3'\ , & x_1x_2x_3-x_{10}x_{11}x_{12}
\end{matrix}\right\}\{-6\}$$
where the polynomials $h_i$ and $h_i'$ ($i=1,2,3$) are as above. The factorizations 
$\zeta(\hat\Gamma_v)$ and $\zeta'(\hat\Gamma_v)$ have potential zero.
Using the explicit form   
$$g_3=h_3=4(x_4+x_5+x_6-a),\qquad g_3'=h_3'=4(x_{10}+x_{11}+x_{12}-a)$$
we exclude the variables $x_4$ and $x_{12}$ from the third and sixth rows in 
$\zeta(\hat\Gamma_v)$ and $\zeta'(\hat\Gamma_v)$. This operation transforms 
the factorization 
$\hat M$ into the factorization $\hat M'$, which again is a tensor factor 
which $\zeta(\hat\Gamma_v)$ and 
$\zeta'(\hat\Gamma_v)$ have in common. Ignoring common overall shifts we obtain
$$\zeta(\hat\Gamma_v) \cong \hat M'\otimes
\left\{\begin{matrix}
g_1\ , & x_1+x_2+x_3-a \\
g_2\ , & x_1x_2+(x_1+x_2)x_3-x_5x_6-(x_5+x_6)(a-x_5-x_6)\\
g_1'\ , & x_7+x_8+x_9-a \\
g_2'\ , & x_8x_9+x_7(x_8+x_9)-x_{10}x_{11}-(x_{10}+x_{11})(a-x_{10}-x_{11})
\end{matrix}\right\}$$
and
$$\zeta'(\hat\Gamma_v) \cong \hat M'\otimes
\left\{\begin{matrix}
h_1\ , & x_7+x_8+x_9-a \\
h_2\ , & x_8x_9+x_7(x_8+x_9)-x_5x_6-(x_5+x_6)(a-x_5-x_6) \\
h_1'\ , & x_1+x_2+x_3-a \\
h_2'\ , & x_1x_2+(x_1+x_2)x_3-x_{10}x_{11}-(x_{10}+x_{11})(a-x_{10}-x_{11})
\end{matrix}\right\}.$$

Using Equation~\ref{eqn:gfactor1} we see that $g_1=h_1'$ and $g_1'=h_1$ and therefore, 
absorbing in $\hat M'$ the corresponding Koszul factorizations, we can write
$$\zeta(\hat\Gamma_v) \cong \hat M''\otimes\hat{K}
\quad\text{and}\quad 
\zeta'(\hat\Gamma_v) \cong \hat M''\otimes\hat{K}'$$
where
$$\hat{K}=\left\{\begin{matrix}
g_2\ , & x_1x_2+(x_1+x_2)x_3-x_5x_6-(x_5+x_6)(a-x_5-x_6)\\
g_2'\ , & x_8x_9+x_7(x_8+x_9)-x_{10}x_{11}-(x_{10}+x_{11})(a-x_{10}-x_{11})
\end{matrix}\right\}$$
and
$$\hat{K}'=\left\{\begin{matrix}
h_2\ , & x_8x_9+x_7(x_8+x_9)-x_5x_6-(x_5+x_6)(a-x_5-x_6) \\
h_2'\ , & x_1x_2+(x_1+x_2)x_3-x_{10}x_{11}-(x_{10}+x_{11})(a-x_{10}-x_{11})
\end{matrix}\right\}.$$

To simplify notation define the polynomials $\alpha_{i,j,k}$ and $\beta_{i,j}$ by
$$
\alpha_{i,j,k} = x_ix_j+(x_i+x_j)x_k, \qquad
\beta_{i,j}    = x_ix_j+(x_i+x_j)(a-x_i-x_j).
$$
In terms of $\alpha_{i,j,k}$ and $\beta_{i,j}$ we have
\begin{equation}\label{eq:swapK1}
\hat{K}=
\left\{\begin{matrix}
2(\alpha_{1,2,3}+\beta_{5,6}) +4b\ , & \alpha_{1,2,3}-\beta_{5,6} \\
2(\alpha_{7,8,9}+\beta_{10,11})+4b\ , & \alpha_{7,8,9}-\beta_{10,11}
\end{matrix}\right\}
\end{equation}
and
\begin{equation}\label{eq:swapK2}
\hat{K}'=
\left\{\begin{matrix}
2(\alpha_{7,8,9}+\beta_{5,6})+4b\ , & \alpha_{7,8,9}-\beta_{5,6} \\
2(\alpha_{1,2,3}+\beta_{10,11})+4b\ , & \alpha_{1,2,3}-\beta_{10,11})
\end{matrix}\right\}.
\end{equation}

Factorizations $\hat{K}$ and $\hat{K}'\brak{1}$ can now be written in matrix form as 
$$\hat{K}=
\begin{pmatrix}R \\ R\end{pmatrix}
\xra{P}
\begin{pmatrix}R \\ R\end{pmatrix}\xra{Q}\begin{pmatrix}R \\ R\end{pmatrix}
\quad\text{and}\quad
\hat{K}'\brak{1}=
\begin{pmatrix}R \\ R\end{pmatrix}
\xra{P'}
\begin{pmatrix}R \\ R\end{pmatrix}\xra{Q'}\begin{pmatrix}R \\ R\end{pmatrix},$$
where
\begin{eqnarray*}
P &=&
\begin{pmatrix}2(\alpha_{1,2,3}+\beta_{5,6})+4b & \alpha_{7,8,9}-\beta_{10,11} \\
2(\alpha_{7,8,9}+\beta_{10,11})+4b & -\alpha_{1,2,3}+\beta_{5,6}\end{pmatrix}
\\
Q &=&
\begin{pmatrix}\alpha_{1,2,3}-\beta_{5,6} & \alpha_{7,8,9}-\beta_{10,11} \\
2(\alpha_{7,8,9}+\beta_{10,11})+4b & -2(\alpha_{1,2,3}+\beta_{5,6})-4b\end{pmatrix}
\end{eqnarray*}
and
\begin{eqnarray*}
P' &=&
\begin{pmatrix}-\alpha_{7,8,9}+\beta_{5,6} & -\alpha_{1,2,3}+\beta_{10,11} \\
-2(\alpha_{1,2,3}+\beta_{10,11})-4b & 2(\alpha_{7,8,9}+\beta_{5,6})+4b\end{pmatrix}
\\
Q' &=&
\begin{pmatrix}
-2(\alpha_{7,8,9}+\beta_{5,6})-4b & -\alpha_{1,2,3}+\beta_{10,11} \\
-2(\alpha_{1,2,3}+\beta_{10,11})-4b & \alpha_{7,8,9}-\beta_{5,6}\end{pmatrix}
\end{eqnarray*}

Define a homomorphism $\psi=(f_0,f_1)$ from $\hat{K}$ to $\hat{K}'\brak{1}$ by the pair of matrices
$$\left(
\begin{pmatrix}
1 & -\frac{1}{2} \vspace{1ex} \\ -1 & -\frac{1}{2}
\end{pmatrix},
\begin{pmatrix}
\frac{1}{2} & -\frac{1}{2} \vspace{1ex}\\ -1 & -1
\end{pmatrix}
\right).$$
It is immediate that $\psi$ is an isomorphism with inverse $(f_1,f_0)$.

It follows that $1_{\hat{M}''}\otimes\psi$ defines an isomorphism between $\zeta(\hat{\Gamma}_v)$ and $\zeta'(\hat{\Gamma}_v)\brak{1}$.
\end{proof}

Although having $\psi$ in this form will be crucial in the proof of Proposition~\ref{prop:ident-vert} an alternative description will be useful in Section~\ref{sec:iso}. Note that we can reduce $\hat K$ and $\hat K'$ in 
Equations~\ref{eq:swapK1} and~\ref{eq:swapK2} further by using the row 
operations $[1,2]_1\circ[1,2]'_{-2}$. We obtain 
$$\hat{K}\cong
\left\{
\begin{matrix}
-4(\alpha_{7,8,9}-\beta_{5,6}), & \alpha_{1,2,3}-\beta_{5,6} \\
2(\alpha_{7,8,9}+\beta_{10,11}+\alpha_{1,2,3}-\beta_{5,6}), & \alpha_{1,2,3}+\alpha_{7,8,9}-\beta_{5,6}-\beta_{10,11} 
\end{matrix}
\right\}$$
and
$$\hat{K}'\cong
\left\{
\begin{matrix}
-4(\alpha_{1,2,3}-\beta_{5,6}), & \alpha_{7,8,9}-\beta_{5,6} \\
2(\alpha_{7,8,9}+\beta_{10,11}+\alpha_{1,2,3}-\beta_{5,6}), & \alpha_{1,2,3}+\alpha_{7,8,9}-\beta_{5,6}-\beta_{10,11} 
\end{matrix}
\right\}.$$
Since the second lines in $\hat{K}$ and $\hat{K}'$ are equal we can 
write 
$$\hat{K}\cong
\left\{
-4(\alpha_{7,8,9}-\beta_{5,6}),\ \alpha_{1,2,3}-\beta_{5,6}\right\}\otimes \hat{K}_2$$
and
$$\hat{K}'\cong
\left\{
-4(\alpha_{1,2,3}-\beta_{5,6}),\ \alpha_{7,8,9}-\beta_{5,6}\right\} \otimes\hat{K}_2.$$

An isomorphism $\psi'$ between $\hat{K}$ and $\hat{K'}\brak{1}$ can now be 
given as the tensor product between $\left( -m(2),\ -m(1/2)\right)$ and the 
identity homomorphism of $\hat{K}_2$. 

\begin{cor}\label{cor:swap}
The homomorphisms $\psi$ and $\psi'$ are equivalent.
\end{cor}

\begin{proof}
The first thing to note is that we obtained the homomorphism 
$\psi$ by first writing the differential $(d_0,d_1)$ in $\hat{K}'$ as 
$2\times 2$ matrices and then its shift $\hat{K}'\brak{1}$ using 
$(-d_1,-d_0)$, but in the computation of $\psi'$ we switched the terms 
and changed the signs in the first line of the Koszul matrix corresponding
 to $\hat{K}'$. The two factorizations obtained are isomorphic by a 
non-trivial isomorphism, which is given by
 $$T=\left(
\begin{pmatrix}
-1 & 0 \\ 0 & 1
\end{pmatrix},\
\begin{pmatrix}
-1 & 0 \\ 0 & 1
\end{pmatrix}
\right).$$
Bearing in mind that $\psi$ and $\psi'$ have $\bZ/2\bZ$-degree 1 and using
$$
\left[1,2\right]_\lambda = \left(\begin{pmatrix}1 &0 \\ 0 & 1\end{pmatrix},\
\begin{pmatrix}1 & -\lambda \\ 0 & 1\end{pmatrix}\right),
\qquad
\left[1,2\right]'_{\lambda} = \left(\begin{pmatrix}1 &0 \\ -\lambda & 1\end{pmatrix},\
\begin{pmatrix}1 & 0 \\ 0 & 1\end{pmatrix}\right),
$$
it is straightforward to check that the composite homomorphism
$T[1,2]_{1}[1,2]'_{-2}\psi[1,2]'_{2}[1,2]_{-1}$ is
$$\left(
\begin{pmatrix}
-2 & 0 \\ 0 & -1/2
\end{pmatrix},\
\begin{pmatrix}
-1/2 & 0 \\ 0 & -2
\end{pmatrix}
\right)$$
which is the tensor product of $\left( -m(2),\ -m(1/2)\right)$ and the 
identity homomorphism of $\hat{K}_2$.
\end{proof}

\begin{proof}[Proof of Proposition~\ref{prop:ident-vert}]
We claim that $\zeta'(\hat\Gamma_v)\cong\zeta(\hat\Gamma_v)\brak{k}$ with $k$ 
a nonnegative integer. We transform $\zeta'(\hat\Gamma_v)$ into 
$\zeta(\hat\Gamma_v)\brak{k}$ by repeated application of Lemma~\ref{lem:swapedges} 
as follows. Choose a pair of vertices connected by a virtual edge in 
$\zeta(\hat\Gamma_v)$. Do nothing if the same pair is connected by a virtual edge 
in $\zeta'(\hat\Gamma_v)$ and use Lemma~\ref{lem:swapedges} to connect them in the 
opposite case. Iterating this procedure we can turn 
$\zeta'(\hat\Gamma_v)$ into $\zeta(\hat\Gamma_v)$ with a shift in the 
$\bZ/2\bZ$-grading by ($k\mod 2$) where $k$ is the number of times we applied 
Lemma~\ref{lem:swapedges}. 

It remains to show that the shift in the $\bZ/2\bZ$-grading is independent of 
the choices one makes. To do so we label the vertices of $\Gamma_v$ of 
($+$)- and ($-$)-type by $(v_1^+,\ldots,v_k^+)$ and 
$(v_1^-,\ldots,v_k^-)$ respectively. Any complete identification of vertices in  
$\Gamma_v$ is completely determined by an ordered set 
$J_\zeta=(v_{\sigma(1)}^-,\ldots ,v_{\sigma(k)}^-)$,  
with the convention that $v_j^+$ is connected through a virtual edge to 
$v_{\sigma(j)}^-$ for $1\leq j\leq k$. Complete identifications of the vertices in 
$\Gamma_v$ are therefore in one-to-one correspondence with the elements of 
the symmetric group on $k$ letters $S_k$. Any transformation of $\zeta'(\hat\Gamma)$ into 
$\zeta(\hat\Gamma)$ by repeated application of Lemma~\ref{lem:swapedges} 
corresponds to a sequence of elementary transpositions whose composite is equal to 
the quotient of the permutations corresponding to $J_{\zeta'}$ and $J_\zeta$. 
We conclude that the shift in the $\bZ/2\bZ$-grading is well-defined, because 
any decomposition of a given permutation into elementary transpositions 
has a fixed parity.
\end{proof}

The next thing to show is that the isomorphisms in the proof of 
Proposition~\ref{prop:ident-vert} 
do not depend on the choices made in that proof. Choose an ordering of the vertices of 
$\Gamma_v$ such that $v^+_i$ is paired with $v^-_i$ for all $i$ and let 
$\zeta$ be the corresponding vertex identification. Use the linear entries in the 
Koszul matrix of $\zeta(\hat{\Gamma}_v)$ to exclude one variable corresponding to an 
edge entering in each vertex of $(-)$-type, as in the proof of 
Lemma~\ref{lem:swapedges}, so that the resulting Koszul factorization 
has the form $\zeta(\hat{\Gamma}_v)=\hat{K}_{lin}\otimes\hat{K}_{quad}$ where 
$\hat{K}_{lin}$ (resp. $\hat{K}_{quad}$) consists of the lines in 
$\zeta(\hat{\Gamma}_v)$ having linear (resp. quadratic) terms as its right entries. 
From the proof of Lemma~\ref{lem:swapedges} we see that changing a pair of 
virtual edges leaves $\hat{K}_{lin}$ unchanged. 

Let $\sigma_i$ be the element of $S_k$ corresponding to the elementary transposition, 
which sends the complete identification $\left(v^-_1,\ldots ,v^-_i,v^-_{i+1},\ldots,v^-_k\right)$ to $\left(v^-_1,\ldots ,v^-_{i+1},v^-_i,\ldots,v^-_k\right)$, 
and let $\Psi_i=1_{/i}\otimes \psi$ be the corresponding isomorphism of matrix 
factorizations from the proof of Lemma~\ref{lem:swapedges}. The homomorphism 
$\psi$ only acts 
on the $i$-th and $(i+1)$-th lines 
in $\hat{K}_{quad}$ and $1_{/i}$ is the identity morphism on the remaining lines. 
For the composition $\sigma_i\sigma_j$ we have the composite homomorphism 
$\Psi_i\Psi_j$.

\begin{lem}\label{lem:repsym}
The assignment $\sigma_i\mapsto \Psi_i$ defines a representation of $S_k$ on 
$\zeta(\hat{\Gamma})_0\oplus\zeta(\hat{\Gamma})_1$.
\end{lem}
\begin{proof}
Let $\hat{K}$ be the Koszul factorization corresponding to the lines $i$ and $i+1$ in $\zeta(\hat{\Gamma}_v)$ and let $\ket{00}$, $\ket{11}$, $\ket{01}$ and $\ket{10}$ be the standard basis vectors of $\hat{K}_0\oplus\hat{K}_1$.
The homomorphism $\psi$ found in the proof of Lemma~\ref{lem:swapedges} can be written 
as only one matrix acting on $\hat{K}_0\oplus\hat{K}_1$:
$$\psi=\begin{pmatrix}
0 & 0 & \frac{1}{2} & -\frac{1}{2}\vspace{1ex} \\
0 & 0 & -1 & -1  \vspace{1ex}\\
1 & -\frac{1}{2} & 0 & 0 \vspace{1ex}\\
-1 & -\frac{1}{2} & 0 & 0
\end{pmatrix}.$$
We have that $\psi^2$ is the identity matrix and therefore it follows that $\Psi_i^2$ 
is the identity homomorphism on $\zeta(\hat{\Gamma}_v)$. It is also immediate that 
$\Psi_i\Psi_j=\Psi_j\Psi_i$ for $|i-j|>1$. To complete the proof we need to show that 
$\Psi_i\Psi_{i+1}\Psi_i=\Psi_{i+1}\Psi_i\Psi_{i+1}$, which we do by explicit 
computation of the corresponding matrices. Let $\hat{K}'$ be the Koszul matrix 
consisting of the three lines $i$, $i+1$ and $i+2$ in $\hat{K}_{quad}$. To show 
that $\Psi_i\Psi_{i+1}\Psi_i=\Psi_{i+1}\Psi_i\Psi_{i+1}$ is equivalent to showing 
that $\psi$ satisfies the Yang-Baxter equation
\begin{equation}\label{eq:YB}
(\psi\otimes 1)(1\otimes\psi)(\psi\otimes 1)=(1\otimes\psi)(\psi\otimes 1)
(1\otimes\psi),
\end{equation}
with $1\otimes\psi$ and $\psi\otimes 1$ acting on $\hat{K}'$. Note that, in general, 
the tensor product of two homomorphisms of matrix factorizations $f$ and $g$ is 
defined by 
$$(f\otimes g)\ket{v\otimes w}=(-1)^{|g||v|}\ket{fv\otimes gw}.$$ Let $\ket{000}$, 
$\ket{011}$, $\ket{101}$, $\ket{110}$, $\ket{001}$, $\ket{010}$, $\ket{100}$ and 
$\ket{111}$ be the standard basis vectors of $\hat{K}'_0\oplus\hat{K}'_1$. 
With respect to this basis the homomorphisms $\psi\otimes 1$ and $1\otimes\psi$ have 
the form of block matrices
$$\psi\otimes 1=
\left(
\begin{array}{c|c}
 0  & 
\begin{array}{cccc}
0 & \frac{1}{2} & -\frac{1}{2} & 0  \vspace{1ex}\\ 
1 & 0     & 0 & -\frac{1}{2}  \vspace{1ex} \\
-1 & 0     & 0 & -\frac{1}{2}  \vspace{1ex}\\
 0  & -1 & -1 & 0  
\end{array}
\\ \hline
\begin{array}{cccc}
0  & \frac{1}{2} & -\frac{1}{2} & 0 \vspace{1ex}\\ 
1 & 0     & 0 & -\frac{1}{2}  \vspace{1ex}\\
-1 & 0     & 0 & -\frac{1}{2}  \vspace{1ex} \\
0  & -1 & -1 & 0
\end{array} & 0
 \end{array}\right)$$
and
$$1\otimes\psi=
\left(
\begin{array}{c|c}
 0  & 
\begin{array}{cccc}
\frac{1}{2} & -\frac{1}{2} & 0 & 0  \vspace{1ex}\\ 
-1 & -1     & 0 & 0  \vspace{1ex} \\
0 & 0     & -1 & \frac{1}{2}  \vspace{1ex}\\
 0  & 0 & 1 & \frac{1}{2}  
\end{array}
\\ \hline
\begin{array}{cccc}
1  & -\frac{1}{2} & 0 & 0 \vspace{1ex}\\ 
-1 & -\frac{1}{2}     & 0 & 0  \vspace{1ex}\\
0 & 0     & -\frac{1}{2} & \frac{1}{2}  \vspace{1ex} \\
0  & 0 & 1 & 1
\end{array} & 0
 \end{array}\right)
.$$ 
By a simple exercise in matrix multiplication we find that both sides 
in Equation~\ref{eq:YB} are equal and it follows that 
$\Psi_i\Psi_{i+1}\Psi_i=\Psi_{i+1}\Psi_i\Psi_{i+1}$.
\end{proof}

\begin{cor}
\label{cor:caniso}
The isomorphism $\zeta'(\hat\Gamma_v)\cong\zeta(\hat\Gamma_v)\brak{k}$ 
in Proposition~\ref{prop:ident-vert} is uniquely determined by $\zeta'$ and $\zeta$. 
\end{cor}
\begin{proof} Let $\sigma$ be the permutation that relates $\zeta'$ and $\zeta$. 
Recall that in the proof of Proposition~\ref{prop:ident-vert} we 
defined an isomorphism $\Psi_{\sigma}\colon\zeta'(\hat\Gamma_v)\to 
\zeta(\hat\Gamma_v)\brak{k}$ by writing $\sigma$ as a product of transpositions. 
The choice of these transpositions is not unique in general. However, 
Lemma~\ref{lem:repsym} shows that $\Psi_{\sigma}$ only depends on $\sigma$.     
\end{proof}

From now on we write $\hat\Gamma$ for the equivalence class of $\hat\Gamma_v$ under 
complete vertex identification. Graphically we represent the vertices of 
$\hat\Gamma$ as 
in Figure~\ref{fig:g-equiv}.
\begin{figure}[ht!]
\labellist
\small\hair 2pt
\pinlabel $\hat\Gamma_v$ at 49 -34
\pinlabel $\hat\Gamma$ at 283 -34
\endlabellist
\centering
\figs{0.3}{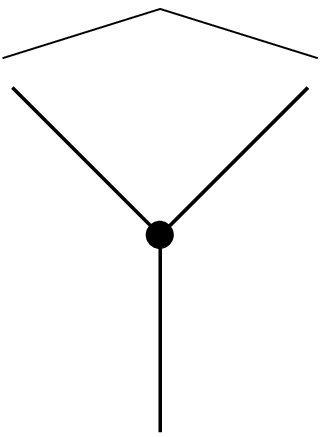}\qquad\qquad
\figs{0.3}{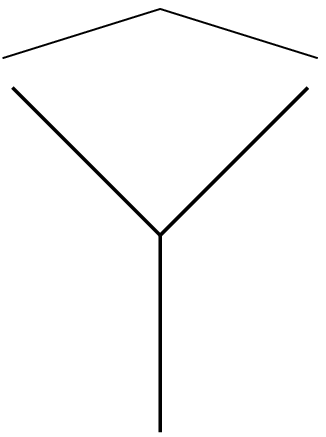}
\bigskip
\caption{A vertex and its equivalence class under vertex identification}
\label{fig:g-equiv}
\end{figure}
\noindent We need to neglect the $\bZ/2\bZ$ grading, which we do by imposing 
that $\hat\Gamma$, for any closed web $\Gamma$, have only homology in degree zero, 
applying a shift 
if necessary.  

We also have to define the morphisms between $\hat\Gamma$ and $\hat\Lambda$. Let 
$\zeta(\hat\Gamma_v)$ and $\zeta'(\hat\Gamma_v)$ be representatives of $\hat\Gamma$ 
and $\zeta(\hat\Lambda_v)$ and $\zeta'(\hat\Lambda_v)$ be representatives of 
$\hat\Lambda$. Let 
$$f\in 
\Hom_{\mf{}}\left(\zeta(\hat\Gamma_v),\zeta(\hat\Lambda_v)\right)$$ 
and $$g\in \Hom_{\mf{}}\left(\zeta'(\hat\Gamma_v),\zeta'(\hat\Lambda_v)\right)$$ 
be two homomorphisms. 
We say that $f$ and $g$ are equivalent, denoted by $f\sim g$, if and only if 
there exists a commuting square
$$\xymatrix@C=15mm{
\zeta(\hat\Gamma_v) 
\ar[d]^f\ar[r]^(0.45){\cong}
& \zeta'(\hat\Gamma_v) 
\ar[d]^g \\
\zeta(\hat\Lambda_v) 
\ar[r]^(0.45){\cong}
& 
\zeta'(\hat\Lambda_v) 
}$$ 
with the horizontal isomorphisms being of the form as discussed in 
Proposition~\ref{prop:ident-vert}. 
The composition rule for these equivalence classes of homomorphisms, which relies 
on the choice of representatives within each class, is well-defined by 
Corollary~\ref{cor:caniso}. Note that we can take well-defined linear combinations of 
equivalence classes of homomorphisms by taking linear combinations of their 
representatives, as long as the latter have all the same source and 
the same target. By Corollary~\ref{cor:caniso}, homotopy equivalences are 
also well-defined on equivalence classes. We take 
$$\Hom(\hat\Gamma,\hat\Lambda)$$
to be the set of equivalence classes of homomorphisms of matrix factorizations 
between $\hat\Gamma$ and $\hat\Lambda$ modulo homotopy equivalence.   
The additive category that we get this way is denoted by 
$$\widehat{\foam}.$$

Note that we can define the homology of ${\hat\Gamma}$, for any 
closed web $\Gamma$. This group is well-defined up to isomorphism and we denote it 
by $\hat\hy(\Gamma)$.

Next we show how to define a link homology using the objects and morphisms in 
$\wfoam$. For any link $L$, first take the universal rational Khovanov-Rozansky 
cochain complex $\KR_{a,b,c}(L)$. The $i$-th cochain group $\KR_{a,b,c}^i(L)$ is 
given by the direct sum of cohomology groups of the form $\hy(\Gamma_v)$, where 
$\Gamma_v$ is a total flattening of $L$. By the remark above it makes sense 
to consider $\widehat{\KR}_{a,b,c}^i(L)$, for each $i$. The differential 
$d^i\colon \KR_{a,b,c}^i(L)\to\KR_{a,b,c}^{i+1}(L)$ induces a map 
$$\hat d^i\colon\widehat{\KR}_{a,b,c}^i(L)\to\widehat{\KR}_{a,b,c}^{i+1}(L),$$
for each $i$. The latter map is well-defined and therefore the homology 
$$\widehat{\HKR}^i_{a,b,c}(L)$$ is well-defined, for each $i$. 

Let $u\colon L\to L'$ be a link cobordism. Khovanov and Rozansky~\cite{KR} constructed 
a cochain map which induces a homomorphism 
$$\HKR_{a,b,c}(u)\colon\HKR_{a,b,c}(L)\to
\HKR_{a,b,c}(L').$$ The latter is only defined up to a $\bQ$-scalar. The induced 
map 
$$\widehat{\HKR}_{a,b,c}(u)\colon\widehat{\HKR}_{a,b,c}(L)\to
\widehat{\HKR}_{a,b,c}(L')$$ is also well-defined up to a $\bQ$-scalar. The 
following result follows immediately:

\begin{lem} $\HKR_{a,b,c}$ and $\widehat{\HKR}_{a,b,c}$ are naturally isomorphic 
as projective functors from $\Link$ to $\Modbg$. 
\end{lem}

In the next section we will show that $U_{a,b,c}$ and $\widehat{\HKR}_{a,b,c}$ 
are naturally isomorphic as projective functors. 

By Lemma~\ref{lem:KR-moy} we also get the following

\begin{lem}\label{lem:KhK-mf}
We have the Khovanov-Kuperberg decompositions in $\wfoam$:
$$
\widehat{\unknot\Gamma} 
\cong
\hatunknot\otimes_{\bQ[a,b,c]}
\hat\Gamma 
\rlap{\hspace{18.2ex}\text{(Disjoint Union)}}
$$
$$
\figins{-9.5}{0.4}{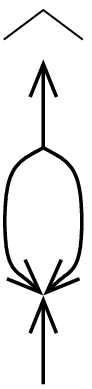}   
\cong 
\figins{-9.5}{0.4}{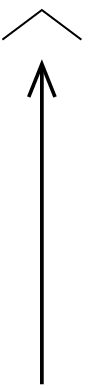}\{-1\}\oplus 
\figins{-9.5}{0.4}{arc-mf}\{1\}
\rlap{\hspace{17.7ex}\text{(Digon Removal)}}
$$
$$
\figins{-8.5}{0.4}{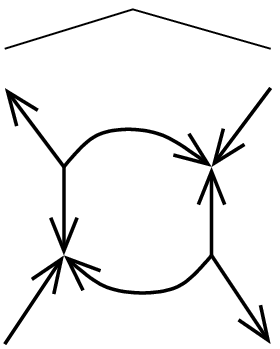}
\cong
\figins{-8.5}{0.405}{sqr1a}\oplus
\figins{-8.5}{0.4}{sqr1b}
\rlap{\hspace{16.3ex}\text{(Square Removal)}}
$$
\end{lem}

\noindent Although Lemma~\ref{lem:KhK-mf} follows from Lemma~\ref{lem:KR-moy} and  
Lemma~\ref{lem:fatedge}, an explicit proof will be useful in the sequel.
\begin{proof}
\emph{Disjoint Union} is a direct consequence of the definitions. To prove
\emph{Digon Removal} define the grading-preserving homomorphisms
$$
\alpha\colon\figins{-9.5}{0.4}{arc-mf}\{-1\}   
\to
\figins{-9.5}{0.4}{digon-mf}
\qquad\qquad
\beta\colon\figins{-9.5}{0.4}{digon-mf}   
\to
\figins{-9.5}{0.4}{arc-mf}\{1\}
$$
by Figure~\ref{fig:dr-mf}.
\begin{figure}[h!]
\labellist
\small\hair 2pt
\pinlabel $\alpha\colon$ at -58 188
\pinlabel $x_1$ at -4 234
\pinlabel $x_2$ at -4 140
\pinlabel $\imath$ at 94 198
\pinlabel $x_1$ at 162 234
\pinlabel $x_2$ at 162 140
\pinlabel -- at 235 184
\pinlabel $x_3$ at 255 184
\pinlabel $x_1$ at 396 234
\pinlabel $x_2$ at 396 140
\pinlabel -- at 398 186
\pinlabel $x_4$ at 380 186
\pinlabel -- at 419 186
\pinlabel $x_3$ at 440 186
\pinlabel $\chi_0$ at 314 199
\pinlabel $\beta\colon$ at -58 48
\pinlabel $x_1$ at -4 94
\pinlabel $x_2$ at -4 0
\pinlabel -- at 0 46
\pinlabel $x_4$ at -18 46
\pinlabel -- at 23 46
\pinlabel $x_3$ at 42 46
\pinlabel $\chi_1$ at 94 59
\pinlabel $x_1$ at 162 94
\pinlabel $x_2$ at 162 0
\pinlabel -- at 235 46
\pinlabel $x_3$ at 255 46
\pinlabel $x_1$ at 396 94
\pinlabel $x_2$ at 396 0
\pinlabel $\epsilon$ at 314 58
\endlabellist
\centering
\figs{0.4}{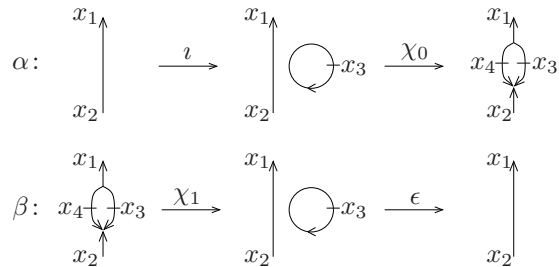}
\caption{Homomorphisms $\alpha$ and $\beta$}
\label{fig:dr-mf}
\end{figure}

\noindent If we choose to create the circle on the other side of the arc in 
$\alpha$ we obtain a homomorphism homotopic to $\alpha$ and the same holds for 
$\beta$. Define the homomorphisms 
$$
\alpha_0\colon
\figins{-9.5}{0.4}{arc-mf}\{-1\}   
\to
\figins{-9.5}{0.4}{digon-mf}
\qquad\qquad
\alpha_1\colon
\figins{-9.5}{0.4}{arc-mf}\{1\}   
\to
\figins{-9.5}{0.4}{digon-mf}
$$
by $\alpha_0=2\alpha$ and $\alpha_1=2\alpha\circ m(-x_2)$. Note that 
the homomorphism $\alpha_1$ is homotopic to the homomorphism 
$2\alpha\circ m(x_1+x_3-a)$.
Similarly define
$$
\beta_0\colon\figins{-9.5}{0.4}{digon-mf}   
\to
\figins{-9.5}{0.4}{arc-mf}\{-1\}
\qquad\qquad
\beta_1\colon\figins{-9.5}{0.4}{digon-mf}   
\to
\figins{-9.5}{0.4}{arc-mf}\{1\}
$$
by $\beta_0=-\beta\circ m(x_3)$ and $\beta_1=-\beta$. A simple calculation shows that $\beta_j\alpha_i=\delta_{ij}\id(\hatarc$).
Since the cohomologies of the factorizations $\hatdigon$ and 
$\hatarc\{-1\}\oplus\hatarc\{1\}$ have the 
same graded dimension (see \cite{KR}) we have that
 $\alpha_0+\alpha_1$ and $\beta_0+\beta_1$ are homotopy inverses of each other and 
that $\alpha_0\beta_0+\alpha_1\beta_1$ is homotopic to the identity in 
$\End(\hatdigon)$.
To prove (\emph{Square Removal}) define grading preserving homomorphisms
$$
\xymatrix@R=3mm{
\psi_0\colon\figins{-9.5}{0.4}{square1-mf}\lra
\figins{-9.5}{0.405}{sqr1a},
&
\psi_1\colon\figins{-9.5}{0.4}{square1-mf}\lra
\figins{-9.5}{0.4}{sqr1b},
\\
\varphi_0\colon\figins{-9.5}{0.4}{sqr1a}\lra
\figins{-9.5}{0.405}{square1-mf},
&
\varphi_1\colon\figins{-9.5}{0.4}{sqr1b}\lra
\figins{-9.5}{0.405}{square1-mf},
}
$$
by the composed homomorphisms below
$$
\xymatrix@C=12mm{
\figins{0}{0.4}{square1-mf}
\ar[r]^{\chi_1\chi_1'} 
\ar@<10pt>@/^1.6pc/[rr]^{\psi_0} &
\figins{0}{0.4}{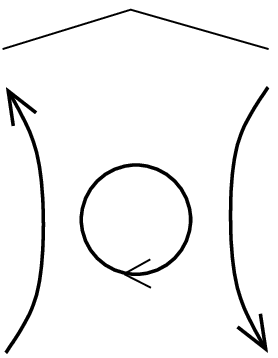}
\ar[r]^{-\varepsilon}
\ar@<6pt>[l]^{\chi_0\chi_0'} &
\figins{0}{0.405}{sqr1a}
\ar@<6pt>[l]^{\imath}
\ar@<10pt>@/^1.6pc/[ll]^{\varphi_0} 
\\
}
\qquad\qquad
\xymatrix@C=12mm{
\figins{0}{0.4}{square1-mf}
\ar[r]^{\overline{\chi}_1\overline{\chi}_1'} 
\ar@<10pt>@/^1.6pc/[rr]^{\psi_1} &
\figins{0}{0.4}{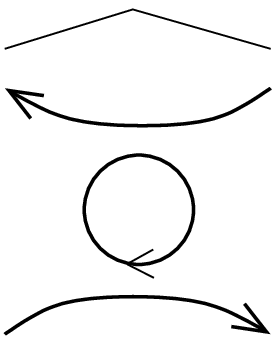}
\ar[r]^{-\varepsilon}
\ar@<6pt>[l]^{\overline{\chi}_0\overline{\chi}_0'} &
\figins{0}{0.405}{sqr1b}
\ar@<6pt>[l]^{\imath}
\ar@<10pt>@/^1.6pc/[ll]^{\varphi_1} 
\\
}.$$
We have that $\psi_0\varphi_0=\id(\hattwoedgesop)$ and 
$\psi_1\varphi_1=\id(\hathtwoedgesop)$. We also have 
$\psi_1\varphi_0=\psi_0\varphi_1=0$ because 
$\Ext(\hattwoedgesop,\hathtwoedgesop)\cong\cH_{a,b,c}(\unknot)\{4\}$ which is zero in 
$q$-degree zero and so any homomorphism of degree zero between $\hattwoedgesop$ and 
$\hathtwoedgesop$ is homotopic to the zero homomorphism. Since the cohomologies of 
$\hatsquare$ and $\hattwoedgesop\oplus\hathtwoedgesop$ have the same graded dimension 
(see \cite{KR}) we have that $\psi_0+\psi_1$ and $\varphi_0+\varphi_1$ are homotopy 
inverses of each 
other and that $\varphi_0\psi_0+\varphi_1\psi_1$ is homotopic to the identity in 
$\End(\hatsquare)$.
\end{proof}

\section{The equivalence functor}\label{sec:iso}

We first define a functor 
$$\widehat{}\,\,\colon \foam\to \widehat{\foam}.$$ 
On objects the functor is defined by $$\Gamma\to\hat\Gamma,$$
as explained in the previous section. 
Let $f\in \Hom_{\foam}(\Gamma,\Gamma')$. Suffice it to consider the case in which 
$f$ can be given by one singular cobordism, also denoted $f$. 
If $f$ is given by a linear combination of 
singular cobordisms, one can simply extend the following arguments to all terms. 
Slice $f$ up between critical 
points, so that each slice contains one elementary foam, i.e. a zip or unzip, 
a saddle-point cobordism, or a cap or a cup, 
glued horizontally to the identity foam on the rest of the source and target webs. 
For each slice choose compatible 
vertex identifications on the source and target webs, such that in the region where 
both webs are isotopic the vertex 
identifications are the same and in the region where they differ the 
vertex identifications are such that we can apply 
the homomorphism of matrix factorizations $\chi_0,\chi_1,\eta, \iota$ or $\epsilon$. 
This way we get a homomorphism of matrix factorizations for each slice. We can take 
its \,\,$\widehat{}$\,\, equivalence class. Composing all these morphisms gives a 
morphism $\hat{f}$ between $\hat\Gamma_v$ and 
$\hat\Gamma'_v$. For its definition 
we had to choose a representative singular cobordism of the foam $f$, a way to slice 
it up and complete vertex identifications for the source and target of each 
slice. Of course we have to show that  
$\hat{f}\in\Hom_{\widehat{\foam}}(\hat\Gamma,\hat\Gamma')$ is independent of those 
choices. But before we do that we have to fullfil a promise that we made at the end 
of Section~\ref{sec:KR-3} after we defined the homomorphism of matrix factorizations 
$\eta$ induced by a saddle-point cobordism. 

\begin{lem}
The map $\hat\eta$ is well defined for closed webs. 
\end{lem}
\begin{proof}
Let $\Gamma$ and $\Gamma'$ be two closed webs and $\Sigma\colon \Gamma\to\Gamma'$ a 
cobordism which is the identity everywhere except for one saddle-point. By a slight 
abuse of notation, let $\hat\eta$ denote the homomorphism of matrix factorizations 
which corresponds to $\Sigma$. Note that $\Gamma$ and $\Gamma'$ have the same number 
of 
vertices, which we denote by $v$. Our proof that $\hat\eta$ is well-defined proceeds 
by induction on $v$. If $v=0$, then the lemma holds, because $\Gamma$ consists of 
one circle and $\Gamma'$ of two circles, or vice-versa. These circles have no marks 
and are therefore indistinguishable. To each circle we associate the 
complex $\hatunknot$ and $\hat\eta$ corresponds to the product or the coproduct in 
$\bQ[a,b,c][X]/{X^3-aX^2-bX-c}$. Note that as soon as we mark the two circles, they 
will become distinguishable, and a minus-sign creeps in when we switch them. 
However, this minus-sign then cancels against the minus-sign 
showing up in the homomorphism associated to the saddle-point cobordism.  

Let $v>0$. This part of our proof uses some ideas from the proof of Theorem 2.4 in 
\cite{JK}. Any web can be seen as lying on a 2-sphere. Let $V,E$ and $F$ denote the 
number of vertices, edges and faces of a web, where a face is a connected component 
of the complement of the web in the 2-sphere. Let $F=\sum_iF_i$, where 
$F_i$ is the number of faces with $i$ edges. Note that we only have faces with an 
even number of edges. It is easy to see that the following 
equations hold:
\begin{eqnarray*}
3V&=&2E\\
V-E+F&=&2\\
2E&=&\sum_iiF_i.
\end{eqnarray*}    
Therefore, we get
$$6=3F-E=2F_2+F_4-F_8-\ldots,$$
which implies 
$$6\leq 2F_2+F_4.$$
This lower bound holds for any web, in particular for $\Gamma$ and $\Gamma'$. 
Therefore we see that there is always a digon or a square in $\Gamma$ and $\Gamma'$ 
on which $\hat\eta$ acts as the identity, i.e. which does not get changed by the 
saddle-point in the cobordism to which $\hat\eta$ corresponds. Since the MOY-moves 
in Lemma~\ref{lem:KhK-mf} are all given by isomorphisms which correspond to the 
zip and the 
unzip and the birth and the death of a circle (see the proof of 
Lemma~\ref{lem:KhK-mf}), this shows that there is 
always a set of MOY-moves which can be applied both to $\Gamma$ and $\Gamma'$ 
whose target webs, say $\Gamma_1$ and $\Gamma'_1$, have less 
than $v$ vertices, and which commute with $\hat\eta$. Here we denote the homomorphism 
of matrix factorizations corresponding to the saddle-point cobordism between 
$\Gamma_1$ and $\Gamma'_1$ by $\hat{\eta}$ again. By induction, 
$\hat\eta\colon\hat{\Gamma}_1\to \hat{\Gamma'}_1$ is well-defined. 
Since the MOY-moves commute with $\hat\eta$, we conclude that 
$\hat\eta\colon\hat\Gamma\to\hat\Gamma'$ is well-defined.  
\end{proof}

\begin{lem}\label{lem:widehat} The functor\,\, $\widehat{}$\,\, is well-defined.  
\end{lem}
\begin{proof} The fact that $\hat{f}$ does not depend on the vertex identifications 
follows immediately from Corollary~\ref{cor:caniso} and the equivalence 
relation $\sim$ on the Hom-spaces in $\widehat{\foam}$.

Next we prove that $\hat{f}$ does not depend on the way we have sliced it up. 
By Lemma~\ref{lem:KhK-mf} we know that, for any closed web $\Gamma$, 
the class $\hat\Gamma$ is homotopy equivalent to a direct sum of terms of the form 
$\hatunknot^k$. Note that $\widehat{\Ext}(\emptyset,\unknot)$ is generated by 
$X^s\iota$, for 
$0\leq s\leq 2$, and that all maps in the proof of Lemma~\ref{lem:KhK-mf} 
are induced by cobordisms with a particular slicing. 
This shows that $\widehat{\Ext}(\emptyset,\Gamma)$ 
is generated by maps of the form $\hat{u}$, where $u$ is a cobordism between 
$\emptyset$ and $\Gamma$ with a particular slicing. A similar result holds for 
$\widehat{\Ext}(\Gamma,\emptyset)$. Now let $f$ and $f'$ be given by the 
same cobordism between $\Gamma$ and $\Lambda$ but with different slicings. If 
$\hat{f}\ne\hat{f'}$, then, by the previous arguments, there exist maps  
$\hat{u}$ and $\hat{v}$, where $u\colon \emptyset\to \Gamma$ and 
$v\colon\Lambda\to\emptyset$ are cobordisms with particular slicings, 
such that $\widehat{vfu}\ne\widehat{vf'u}$. 
This reduces the question of independence of 
slicing to the case of closed cobordisms. Note that we already know that 
\,\,$\widehat{}$\,\, is well-defined on the parts that do not involve 
singular circles, because it is the generalization of a 2d TQFT. 
It is therefore easy to see that \,\,$\widehat{}$\,\, respects the 
relation (CN). Thus we can cut up any closed singular cobordism near the 
singular circles to obtain a linear combination of closed singular 
cobordisms isotopic to spheres and theta-foams. The spheres do not have 
singular circles, 
so \,\,$\widehat{}$\,\, is well-defined on them and it is easy to check 
that it respects the relation (S). 

Finally, for theta-foams we do have to 
check something. There is one basic Morse move that can be applied to 
one of the discs of a theta-foam, which we show in 
Figure~\ref{fig:SingMorse}. We have to show that \,\,$\widehat{}$\,\, 
is invariant under this Morse move. 

\begin{figure}[ht!]
\centering
\figins{0}{0.65}{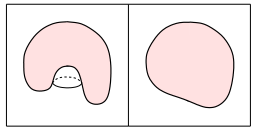}
\caption{Singular Morse move}
\label{fig:SingMorse}
\end{figure}

\begin{figure}[ht!]
\centering
\labellist
\small\hair 2pt
\pinlabel $x_1$ at  -2 80
\pinlabel $x_2$ at  45 38
\pinlabel $x_1$ at  92 80
\pinlabel $x_2$ at 141 26
\pinlabel $x_2$ at 225 20
\pinlabel $x_3$ at 206 10
\pinlabel $x_2$ at 313 16
\pinlabel $x_3$ at 295 10
\pinlabel $x_3$ at 385 12
\endlabellist
$$
\figins{-28}{0.77}{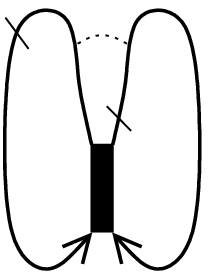}\xra{\chi_0}
\figins{-28}{0.77}{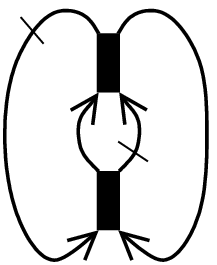}\xra{\psi}
\figins{-28}{0.77}{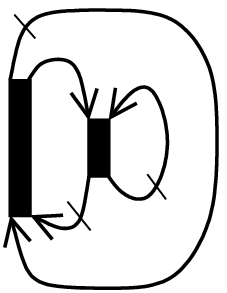}\xra{\chi_1}
\figins{-28}{0.77}{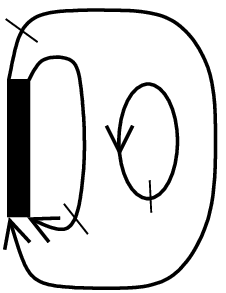}\xra{\varepsilon}
\figins{-28}{0.77}{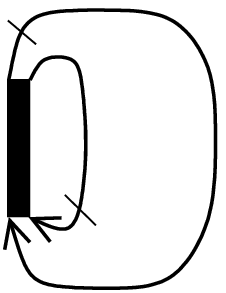}\xra{\tilde{\psi}}
\figins{-28}{0.77}{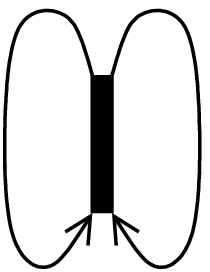}
$$
\caption{Homomorphism $\Phi$. To avoid cluttering only some marks are shown}
\label{fig:smorse}
\end{figure}
\n In other words, we have to show that the composite homomorphism in 
Figure~\ref{fig:smorse} is homotopic to the identity. It suffices to do 
the computation on the homology. First we note that the theta web has 
homology only in $\bZ/2\bZ$-degree $0$. From the remark at the end of  
Subsection~\ref{ssec:diffs} it follows that $\chi_0$ is equivalent to  
multiplication by $-2(x_1-x_2)$ and $\chi_1$ to multiplication by 
$x_2-x_3$, where we used the fact that $\psi$ has $\bZ/2\bZ$-degree $1$. 
From Corollary~\ref{cor:swap} we have that $\psi$ is equivalent to 
multiplication by $-2$ and from the definition of vertex identification it 
is immediate that $\tilde{\psi}$ is the identity. Therefore we have that
$$\Phi=\varepsilon\left(4(x_2-x_3)(x_1-x_2)\right)=1.$$
It is also easy to check that \,\,$\widehat{}$\,\, respects the relation 
($\Theta$). 

Note that the arguments above also show that, for an open foam $f$, 
we have $\hat{f}=0$ if $u_1fu_2=0$ for all singular cobordisms 
$u_1\colon \emptyset\to \Gamma_v$ and $u_2\colon\Gamma'_v\to\emptyset$. 
This proves that \,\,$\widehat{}$\,\, is well-defined on foams, which are 
equivalence classes of singular cobordisms. 
\end{proof}

\begin{cor}\label{cor:widehat}
\,\,$\widehat{}$\,\, is an isomorphism of categories.  
\end{cor}
\begin{proof} On objects \,\,$\widehat{}$\,\, is clearly a bijection. On morphisms it 
is also a bijection by Lemma~\ref{lem:KhK-mf} and the proof of Lemma~\ref{lem:widehat}.
\end{proof}

\begin{thm}
The projective functors $U_{a,b,c}$ and $\widehat{\HKR}_{a,b,c}$ from $\Link$ to 
$\Modbg$ are naturally isomorphic. 
\end{thm}
\begin{proof}
Let $D$ be a diagram of $L$, $C_{\foam}(D)$ the complex for $D$ constructed with 
foams in Section~\ref{sec:foam} and $\widehat{KR}_{a,b,c}(D)$ the complex constructed with 
equivalence classes of matrix factorizations in Section~\ref{sec:sl3-mf}.
From Lemma~\ref{lem:widehat} and Corollary~\ref{cor:widehat} it follows that for 
all $i$ we have isomorphisms of graded $\bQ[a,b,c]$-modules 
$C^i_{\foam}(D)\cong \widehat{KR}^i_{a,b,c}(D)$ where $i$ is 
the homological degree. By a slight abuse of notation we denote these isomorphisms 
by $\,\,\widehat\,\,$ too. The differentials in $\widehat{KR}_{a,b,c}(D)$ are 
induced by $\chi_0$ and $\chi_1$, which are exactly the maps that we associated to 
the zip and the unzip. This shows that $\,\,\widehat{}\,\,$ commutes with the 
differentials in both complexes and therefore that it defines an isomorphism of 
complexes.  

The naturality of the isomorphism between the two functors follows from 
Corollary~\ref{cor:widehat} and the fact that all elementary link cobordisms 
are induced by the elementary foams and their respective images w.r.t.  
$\,\,\widehat\,\,$. 
\end{proof}


\vspace*{1cm}

\noindent {\bf Acknowledgements} The authors thank Mikhail Khovanov for interesting 
conversations and enlightening exchanges of e-mail about the topic of this paper.

Both authors were supported by the 
Funda\c {c}\~{a}o para a Ci\^{e}ncia e a Tecnologia through the
programme ``Programa Operacional Ci\^{e}ncia, Tecnologia, Inova\c
{c}\~{a}o'' (POCTI), cofinanced by the European Community fund FEDER.


\end{document}